\documentclass{amsart}%
\usepackage[T1]{fontenc}
\usepackage{amsfonts}
\usepackage{mathtools,graphicx}
\usepackage{latexsym}
\usepackage{hyperref}
\usepackage{xcolor}
\usepackage{amsmath}
\usepackage{mathrsfs}
\usepackage{listings}
\usepackage{hyperref}
\usepackage{enumitem}
\usepackage{csvsimple}
\usepackage{todonotes}
\usepackage{amssymb,bm}
\usepackage{amsmath}
\usepackage{amsthm}
\usepackage{tikz}
\usepackage{tikz-cd}
\usepackage{epsfig}
\usepackage{verbatim}
\usepackage{float}
\usepackage{tabularx}
\usepackage{algorithmic}
\usepackage{amssymb}
\usepackage[normalem]{ulem}
\usepackage{graphicx}%
\providecommand{\U}[1]{\protect\rule{.1in}{.1in}}
\definecolor{codegray}{rgb}{0.5,0.5,0.5}
\lstdefinestyle{mystyle}{
commentstyle=\color{codegreen},
keywordstyle=\color{magenta},
numberstyle=\tiny\color{codegray},
stringstyle=\color{codepurple},
basicstyle=\ttfamily\footnotesize,
breakatwhitespace=false,
breaklines=true,
captionpos=b,
keepspaces=true,
numbers=left,
numbersep=5pt,
showspaces=false,
showstringspaces=false,
showtabs=false,
tabsize=2
}
\lstdefinelanguage{GAP}{
basicstyle=\ttfamily,
keywords={true, false, function, return, fail, if, in, while, do, od, else, elif, fi, break, continue},
keywordstyle=\color{blue}\bfseries,
otherkeywords={      >, <, ==
},
breaklines=true,
identifierstyle=\color{black},
sensitive=True,
comment=[l]{\#},
commentstyle=\color{cyan},
stringstyle=\color{red},
morestring=[b]',
morestring=[b]"
}
\providecommand{\U}[1]{\protect\rule{.1in}{.1in}}

\newcolumntype{Y}{>{\raggedleft\arraybackslash}X}

\def\bz{{\mathbb{Z}}}

\def\vs{\vskip.3cm}
\def\noi{\noindent}

\def\gdeg{G\text{\rm -deg}}

\def\Om{\Omega}

\newtheorem{theorem}{Theorem}

\newtheorem{corollary}[theorem]{Corollary}

\newtheorem{definition}[theorem]{Definition}
\newtheorem{example}[theorem]{Example}

\newtheorem{lemma}[theorem]{Lemma}

\newtheorem{proposition}[theorem]{Proposition}
\newtheorem{remark}[theorem]{Remark}

\def\bz{\mathbb Z}
\def\vs{\vskip.3cm}

\def\vp{\varphi}

\def\bz{{\mathbb{Z}}}

\def\vs{\vskip.3cm}
\def\noi{\noindent}

\def\gdeg{G\text{\rm -deg}}

\def\Om{\Omega}

\def\cV{\mathcal V}
\def\cW{\mathcal W}

\def\cU{\mathcal U}

\def\id{\text{\rm Id}}

\begin{document}
\title[Bifurcation in four-component Bose-Einstein condensates]{Global bifurcation in four-component Bose-Einstein condensates in space}
\author[C. Garc\'ia-Azpeitia]{Carlos Garc\'ia-Azpeitia}
\address{
Department of Mathematics and Mechanics IIMAS, UNAM, Universidad 3000, 04510.
Ciudad de M\'exico, M\'exico.}
\email{cgazpe@aries.iimas.unam.mx}
\author[A. Go\l \c{e}biewska]{Anna Go\l \c{e}biewska}
\address{Faculty of Mathematics and Computer Science, Nicolaus Copernicus University in
Toru\'n, ul. Chopina 12/18, 87-100 Toru\'n, Poland}
\email{aniar@mat.umk.pl}
\author[W. Krawcewicz]{Wieslaw Krawcewicz}
\address{Department of Mathematical Sciences, the University of Texas at Dallas,
Richardson, 75080 USA}
\email{wieslaw@utdallas.edu}
\author[J. Liu]{Jingzhou Liu}
\address{Department of Mathematical Sciences, the University of Texas at Dallas,
Richardson, 75080 USA}
\email{Jingzhou.Liu@UTDallas.edu}

\begin{abstract}
We study a system of coupled Gross--Pitaevskii equations modelling a
four-component Bose--Einstein condensate on the unit ball in $\mathbb{R}^{3}$.
 The system possesses a branch
of spatially homogeneous solutions---the \emph{trivial branch}---that
bifurcates from the identically zero state. Our main result establishes that
this trivial branch undergoes multiple global bifurcations at critical values
of the chemical potential. The  tool we use to prove our results is the  $G$-equivariant gradient degree with $G=O(3)\times
S_{4}$, which combines the rotational symmetry of the domain with the
permutation symmetry of the components.  We emphasise that the bifurcating solutions are spatially
inhomogeneous \emph{amplitude} (density) patterns of real-valued components.
\end{abstract}\vs
\maketitle
\noi \textbf{Mathematics Subject Classification:} Primary: 35Q55, 37G40, Secondary: 20C15, 55M20, 58C27

\medskip

\noi \textbf{Key Words and Phrases:} Gross–Pitaevskii equation; bifurcation with symmetries; equivariant degree; irreducible representations.

\section{Introduction}

\subsection*{The model and the bifurcation problem}

This paper concerns the global bifurcation of stationary states of a
system of coupled Gross--Pitaevskii equations modelling a multi-component
Bose--Einstein condensate (BEC) in a box-like confinement. In the Gross--Pitaevskii framework, a BEC comprising $n$
distinct gas components is described by a vector of wave functions
$u=(u_{1},\dots,u_{n}):\Omega\rightarrow\mathbb{C}^{n}$, defined on a bounded
domain $\Omega\subset\mathbb{R}^{3}$, where $|u_{j}|^{2}$ represents the
density of the $j$-th component. The interactions are intra-component
(within a single component, of strength $g_{ii}$) and inter-component (between
distinct components, described by the symmetric coupling constants $g_{ij}=g_{ji}$, $i\neq j$). The energy
functional in the mean-field limit is
\begin{equation}
E(u)=\frac{1}{2}\int_{\Omega}\sum_{j=1}^{n}\left(  \left\vert \nabla
u_{j}\right\vert ^{2}-\mu_{j}\left\vert u_{j}\right\vert ^{2}\right)
+\sum_{i,j=1}^{n}\frac{g_{ij}}{2}\left\vert u_{i}\right\vert ^{2}\left\vert
u_{j}\right\vert ^{2}~dx\text{,} \label{energy}%
\end{equation}
where $\mu_{j}$ is the chemical potential of the $j$-th component. 

The stationary states of the condensate are critical points of the energy
$E(u)$, which satisfy a system of coupled Gross--Pitaevskii equations subject
to Neumann boundary conditions. The Neumann condition corresponds to a
box-type confinement in which the density need not vanish at the boundary, the
regime appropriate for the bulk amplitude bifurcations studied here.

In our paper we consider the symmetric mixture, therefore we assume $\mu_{j}=\mu$,
$g_{jj}=g$ and $g_{ij}=\widetilde{g}$ for $i\neq j$. In this setting, it is known that  the system admits a
branch of spatially homogeneous solutions---the \emph{trivial branch}---that
bifurcates from the identically zero state $u\equiv 0$ (the \emph{zero solution}):
\begin{equation}
u_{\mu}=(c_{\mu},\dots,c_{\mu}),\qquad c_{\mu}=\left(  \frac{\mu
}{g+(n-1)\tilde{g}}\right)  ^{1/2}. \label{eq:c}%
\end{equation}
The aim of this
work is to study the existence of \textit{non-trivial} solutions---spatially inhomogeneous amplitude
patterns---arising from $u_{\mu}$ via secondary bifurcation, using the
chemical potential $\mu\in\mathbb{R}$ as the bifurcation parameter. In this
picture the trivial branch $u_{\mu}$ constitutes the primary bifurcation from
the zero solution, and the non-trivial branches that we seek are secondary bifurcations.

The mathematical analysis of coupled nonlinear Schr\"odinger and
Gross--Pitaevskii systems of the form~\eqref{energy} has been the subject of
extensive rigorous study over the past two decades. For coupled systems
on $\mathbb{R}^{n}$, Lin and Wei~\cite{LinWei2005} studied ground states of
$N$ coupled equations in dimension $n\le3$ and showed that their existence
depends on the signs of the coupling constants, while Ambrosetti and
Colorado~\cite{AmbrosettiColorado2007} proved the existence of bound and
ground states for small and for large values of the coupling parameter, respectively. A continuum of bound states bifurcating
from semi-trivial solutions was obtained by Bartsch, Wang, and
Wei~\cite{BartschWangWei2007} using fixed-point index theory in cones, and
Bartsch, Dancer, and Wang~\cite{BartschDancerWang2010} combined Liouville
theorems with a priori bounds to produce bifurcating branches of positive
solutions; the interplay between a priori bounds and multiplicity was further
investigated by Dancer, Wei, and Weth~\cite{DancerWeiWeth2010}. Radial
solutions and phase separation for two coupled equations were studied by Wei
and Weth~\cite{WeiWeth2008}, and multiple bound states were proved via
variational and topological methods by Liu and Wang~\cite{LiuWang2008}.

A central phenomenon captured by these models is phase separation---the
spatial segregation of components as the inter-component repulsion increases.
Multipulse phases in $k$-mixtures were obtained by Terracini and
Verzini~\cite{TerraciniVerzini2009}. For systems with mixed attractive and
repulsive couplings, Soave~\cite{Soave2015} established the existence of
solutions with at least $m$ positive components, for every $m$ not exceeding the number of components, and Byeon, Sato,
and Wang~\cite{ByeonSatoWang2016,ByeonSatoWang2017} proved pattern formation
results under both Dirichlet and Neumann boundary conditions. Normalized
solutions---the natural setting when particle numbers are fixed---have been
studied by Bartsch, Jeanjean, and Soave~\cite{BartschJeanjeanSoave2016} and
Bartsch, Zhong, and Zou~\cite{BartschZhongZou2021}.

Turning to bifurcation-theoretic results more closely related to our work,
Chang, Lin, Lin, and Lin~\cite{ChangLinLinLin2004} analysed segregated nodal
domains in two-dimensional multispecies condensates, and the bifurcation
structure of two-component condensates---including supercritical pitchfork
bifurcations at finite inter-component scattering lengths---was analysed by Kuo, Lin, and Shieh~\cite{KuoLinShieh2005}. Symmetry-breaking bifurcations for a
single-component Gross--Pitaevskii equation in a double-well potential were
studied by Kirr, Kevrekidis, Shlizerman, and
Weinstein~\cite{KirrKevrekidis2008}. In the single-component setting with
rotational symmetry, the bifurcation of multi-vortex configurations from the
zero solution has been treated by Contreras and
Garc\'{\i}a-Azpeitia~\cite{CoGa15} and Garc\'{\i}a-Azpeitia and
Pelinovsky~\cite{GaPe17} using degree methods. Ground states of two-component
condensates coupled by a Josephson junction (Rabi coupling) were studied by Peng, Wang, Zeng, and
Zhou~\cite{2layer}, who analysed the existence of minimizers at the endpoint values of the coupling parameters.
The present paper extends this line of work to a four-component system in
which the high-dimensional kernels arising from the combined spatial and
permutation symmetry necessitate the full $G$-equivariant gradient degree machinery.

 \subsection*{Symmetry and the choice of four components}
In our paper we consider the four-component condensate. This choice is dictated by the richness of the
symmetric solutions detected by equivariant-degree methods. Four components
produce solutions with polyhedral symmetry groups, and not merely the dihedral
groups that occur for $n=3$ (cf.\ \cite{MM2}). However, the methods developed here can be
adapted to other numbers of components and to other domains.

The equal self-interaction assumption $g_{jj}=g$ and the equal chemical
potential assumption $\mu_{j}=\mu$ define a natural $S_{4}$-invariant manifold
in parameter space. On this manifold the Hessian of the energy at $u_{\mu}$ has a
block structure determined by the permutation symmetry: the interaction matrix
is diagonalisable with a small number of distinct eigenvalues, making the
isotypic decomposition explicit and the computation of the equivariant
gradient degree tractable. Equal chemical potentials correspond to
equal Lagrange multipliers enforcing equal particle-number constraints in the
variational formulation.

We note that the methods can in principle accommodate non-identical
self-interaction coefficients $g_{ii}$ and inter-component couplings $g_{ij}$,
albeit at the cost of a more delicate analysis of bifurcation points and
non-degeneracy conditions. Small deviations from exact symmetry, of the form
$\mu_{i}=\mu+\delta\mu_{i}$ with $|\delta\mu_{i}|$ small, break the $S_{4}$
permutation symmetry and generically split degenerate eigenvalues. However,
the equivariant gradient degree is a topological invariant: a nonzero degree
change at a symmetric bifurcation point guarantees the existence of a global
branch of solutions, and small continuous perturbations that do not introduce
additional critical parameter values cannot annihilate such a branch---they
can only deform it. Consequently, the existence of non-trivial amplitude
branches is robust under the small asymmetries that are inevitable in
experiment, even though the symmetry classification of those branches may change.

We emphasise that we study \emph{amplitude} (density) secondary bifurcations,
not vortices. %\sout{The restriction to real-valued solutions is a deliberate choice
%guided by the bifurcation analysis, which shows that only real solutions exist
%as secondary bifurcations through the trivial branch $u_{\mu}$.} 
The analysis of the kernel of the Hessian at critical value shows that only real 
solutions exist as secondary bifurcations from the trivial branch $u_{\mu}$. Therefore we restrict our considerations
to real-valued functions.
Note that vortex
nucleation requires a complex-valued order parameter carrying a non-trivial
winding phase (topological charge), which is not present in our setting. One
could carry out similar equivariant degree methods to find bifurcations of
vortices and dipole-like solutions that bifurcate from the zero solution
\thinspace$u=0$ at different values of $\mu$, as it has been done in Contreras
and Garc\'{\i}a-Azpeitia~\cite{CoGa15}, and Garc\'{\i}a-Azpeitia and
Pelinovsky~\cite{GaPe17} with an additional rotation of frequency.

Within this symmetric framework, we treat four real-valued components on the
unit ball $\Omega\subset\mathbb{R}^{3}$, i.e., $u:\Omega\rightarrow
\mathbb{R}^{4}$. The
system of Gross--Pitaevskii equations $\nabla_{u}E(u)=0$
then takes the explicit form
\begin{equation}
-\Delta u_{j}-\mu u_{j}+gu_{j}^{3}+\sum_{\substack{i=1\\i\neq j}}^{4}\tilde
{g}\,u_{i}^{2}u_{j}=0,\qquad\left.  \frac{\partial u_{j}}{\partial\nu
}\right\vert _{|x|=1}=0,\qquad j=1,\dots,4. \label{GPE}%
\end{equation}

 \subsection*{Main result}
We now state our main result. On the
space of functions $u:\Omega\rightarrow\mathbb{R}^{4}$, we can consider the
orthogonal action of the group
\[
G=O(3)\times S_{4},
\]
where $S_{4}$ is the group of permutations of $\left\{  1,2,3,4\right\}  $,
defined by
\begin{equation}\label{eq: actions}
\begin{cases}
\gamma u(x)=u(\gamma^{-1}x),\;\;\gamma\in O(3),\\
\sigma u(x)=(u_{\sigma^{-1}(1)},...,u_{\sigma^{-1}(4)}),~\sigma\in S_{4}.
\end{cases}
\end{equation}
One readily verifies that the energy $E$ is $G$-invariant. Consequently, for
any solution $u$ of \eqref{GPE} and $\mathfrak{g} \in G$, the function $\mathfrak{g}u$ is also a
solution of \eqref{GPE}, i.e. the solutions of \eqref{GPE} form $G$-orbits.

Throughout this paper we impose the \emph{immiscible} assumption
\begin{equation}
0<g<\tilde{g}. \label{eq:immiscible}
\end{equation}
Under this condition the trivial solution $u_{\mu}$ is defined for $\mu\in
\mathbb{R}^{+}$.

Let $\mathcal K:=\{(0,0)\}\cup \{(k,m)\in \mathbb N\times \mathbb N^+\}$ and denote by $\sigma(-\triangle)=\{s_{km}^{2}:(k,m)\in\mathcal K\}$ the spectrum of
the Laplace operator $-\triangle$ with Neumann boundary conditions in $\Omega$
(see Appendix \ref{app:laplace} for details). To each eigenvalue
$s_{km}^{2}$ there is associated an $O(3)$-irreducible representation with
dimension $2k+1$. In particular, the description of numbers $s_{km}$ and corresponding eigenspaces is given
in \eqref{eq:spL} and \eqref{eigenspace}. Define
\[
\mu_{km}:=\frac{g+3\tilde{g}}{2\left(  \tilde{g}-g\right)  }s_{km}^{2}.
\]
We say the critical value $\mu_{km}$ is \emph{isotypic simple} in the case
that $s_{km}^{2}\neq s_{jn}^{2}$ for all $(j,n)\neq(k,m)$.

The aim of our paper is to study the existence of non-trivial solutions of
\eqref{GPE}. More precisely, we investigate the phenomenon of global
bifurcation from the trivial branch  $u_{\mu}$. Our main result, concerning global bifurcation in system
\eqref{GPE} is Theorem \ref{Thm3}. This theorem in its precise form requires isotypic simple
critical values. %\sout{Notice that a study of bifurcation can be carried out on the zero solution
%$u=0$, except that in such a case the critical values are never isotypic simple.}
 Below we present the result and discuss the methods that we use in the proof.

\begin{theorem}
\label{th:main1} Under the \textit{immiscible }assumption $g<\tilde{g}$, the
Gross--Pitaevskii equation \eqref{GPE} undergoes a global bifurcation from the
trivial solution $u_{\mu}$ at any critical value $\mu_{km}$ with
$(k,m)\neq(0,0)$. In the case $k=0$ there exist at least two global branches of radially symmetric solutions. Furthermore, in the case $k=1$, we prove that there are at least six
$G$-orbits of global branches with different symmetries bifurcating from $u_{\mu}$ at the isotypic
simple critical value $\mu_{1m}$ for $m \in\mathbb{N}^{+}$; three of these symmetries are illustrated in Figure~\ref{fig:symmetries}.
\end{theorem}

This theorem is an immediate consequence of our Theorems \ref{Thm3} and
\ref{Thm4}. The difficulty in proving this theorem resides in the fact that for the
isotypic simple critical value $\mu_{km}$ the Hessian of the energy has a kernel equivalent
to a $G$-irreducible representation of dimension $3(2k+1)$. Additionally,
multiple irreducible representations can appear in the kernel when different
eigenvalues coincide. Thus the bifurcation problem has an undetermined high
dimensional kernel. It is important to note that if an even number of
eigenvalues coincide, and the gradient structure is disrupted by a
perturbation, the theorem becomes invalid and there can be critical values
that are not bifurcation points.

Figure~\ref{fig:symmetries} illustrates three of the symmetry types of the branches given by Theorem~\ref{th:main1} for $k=m=1$. It indicates the regions where the mass of each component is concentrated. Each panel shows, for a subgroup $H$, the first-order perturbation $u_H=u_{\mu_{11}}+\varepsilon\Phi_H$ of the trivial solution, where $\Phi_H$ is the function with symmetries $H$ computed in Section~\ref{sec:visualization}.

\begin{figure}[htbp]
    \centering
    \begin{minipage}[t]{0.32\textwidth}
        \centering
        \includegraphics[width=\linewidth,trim=102 102 102 102,clip]{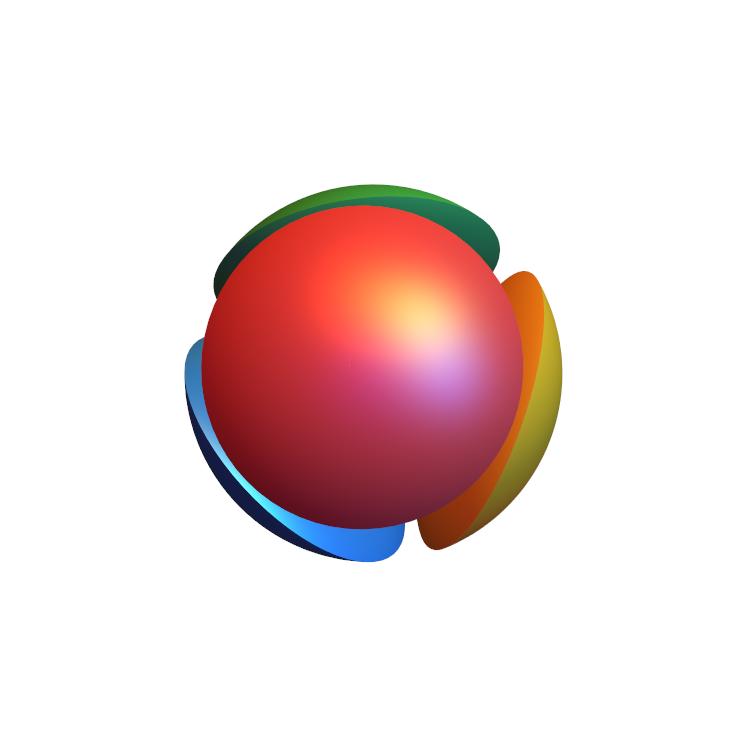}\\[1mm]
        (a) $(S_4^-\times_{S_4}S_4)$
    \end{minipage}\hfill
    \begin{minipage}[t]{0.32\textwidth}
        \centering
        \includegraphics[width=\linewidth,trim=102 102 102 102,clip]{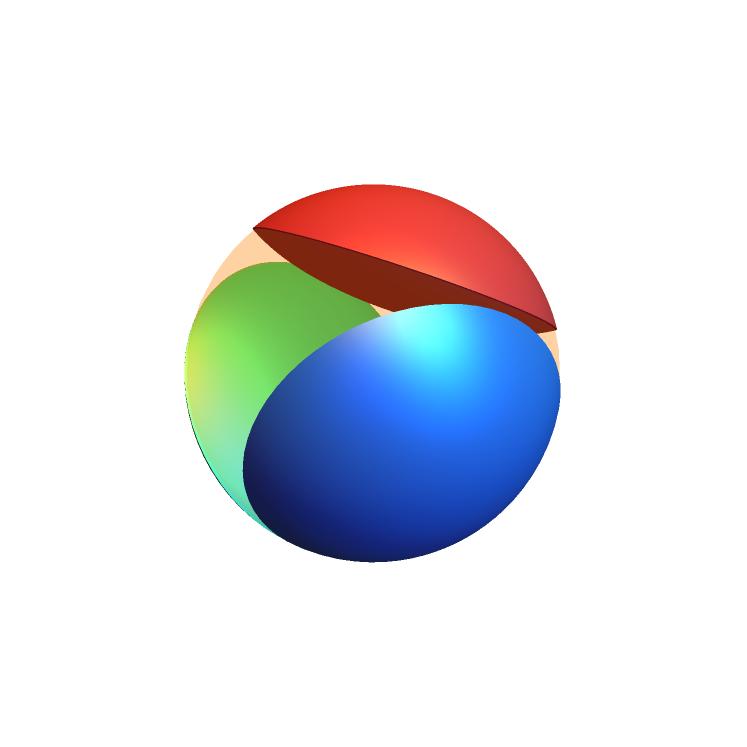}\\[1mm]
        (b) $(D_3\times_{D_3}D_3)$
    \end{minipage}\hfill
    \begin{minipage}[t]{0.32\textwidth}
        \centering
        \includegraphics[width=\linewidth,trim=102 102 102 102,clip]{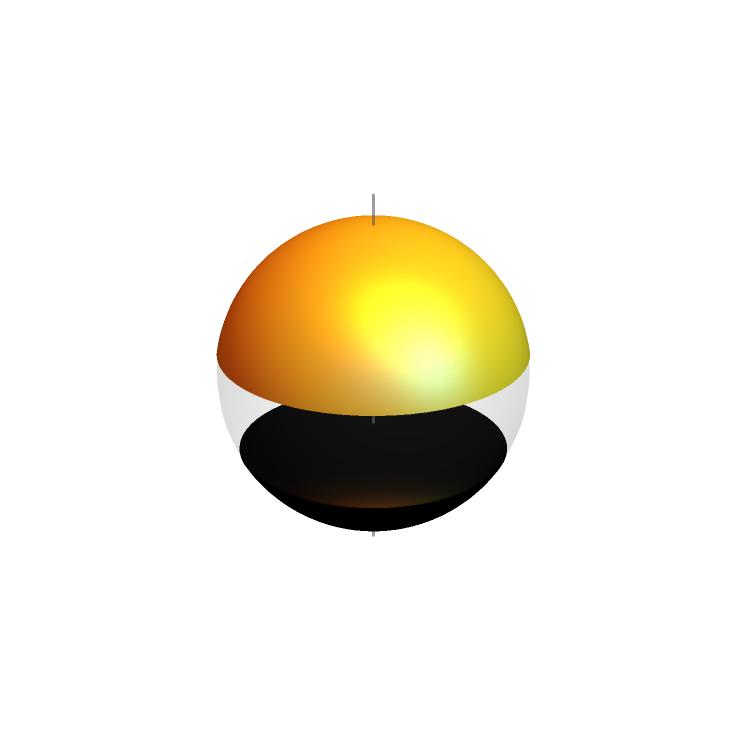}\\[1mm]
        (c) $(O(2)^-\times D_3)$
    \end{minipage}
    \caption[Symmetry types of the branches for $k=m=1$]{Three symmetry types for $k=m=1$. (a) Each of the four colours is a component, located around a different vertex of a regular tetrahedron. (b) Three components are located near the vertices of an equilateral triangle; the fourth component is uniform throughout the ball and appears as a light orange background. (c) Three components are located in the black region at the bottom of the ball and the fourth component in the yellow region at the top; the configuration is symmetric under rotations about the $x_3$-axis.}
    \label{fig:symmetries}
\end{figure}

 \subsection*{$G$-equivariant gradient degree}

The main tool that we use to prove our result is the theory of $G$-equivariant gradient degree. We recall some of the properties and computation of this invariant in
Appendix \ref{app:gradient}. The change of this degree, when the parameter
$\mu$ crosses the value $\mu_{km}$, implies the global bifurcation. Since this
degree is an element of the Euler ring $U(G)$ (see Appendix \ref{app:Eulerring}
for the definition), its coefficients correspond to conjugacy classes of
closed subgroups of the group $G$. In particular, if the change of the degree
occurs for the coefficient corresponding to the maximal orbit type, we obtain
the global bifurcation of the branch of solutions with given symmetry. This
allows us to formulate the multiplicity result for bifurcating branches. In this
way, the theory of $G$-equivariant gradient degree provides a topological
description of the set of solutions of \eqref{GPE}, which cannot be obtained
by other methods.

On the other hand, the rich symmetries of the system make the problem very
difficult to study with $G$-equivariant degree methods, since the structure of
the Euler ring is not fully described in this case. However, one can reduce
the computations to the ones for the Brouwer $G$-equivariant degree, which
takes the values in the Burnside ring. For such degree we can compute the
values of the coefficients. Then, using the relation between $G$-equivariant
gradient degree and the Brouwer $G$-equivariant degree, we obtain the proof of
our result.

The paper is structured as follows. In Section 2 we introduce the essential
mathematical framework necessary for applying the equivariant degree. In
Section 3 we discuss the spectral properties of the linearized equation around
the trivial solution. In Section 4 we present the proof of the main Theorem
\ref{th:main1}, utilizing the equivariant gradient degree defined in Appendix
\ref{app:gradient}. Section \ref{sec:visualization} describes the symmetries illustrated in Figure~\ref{fig:symmetries} and Section \ref{sec:future} contains concluding remarks and a discussion of possible extensions.

\section{Mathematical framework in functional space}

Observe that under the assumptions that $n=4$, $\mu_{j}=\mu$, $g_{jj}=g$ and
$g_{ij}=\widetilde{g}$ for $i\neq j$, and real functions $u:\Omega
\rightarrow\mathbb{R}^{4}$, the energy (\ref{energy}) $E_{\mu}:H^{2}
(\Omega,\mathbb{R}^{4})\rightarrow\mathbb{R}$ becomes explicitly
\begin{equation}
E_{\mu}(u)=\frac{1}{2}\int_{\Omega}\sum_{j=1}^{4}\left(  \left\vert \nabla
u_{j}\right\vert ^{2}-\mu u_{j}^{2}+\frac{g}{2}u_{j}^{4}\right)
+\sum_{i,j=1(i\neq j)}^{4}\frac{\tilde{g}}{2}u_{i}^{2}u_{j}^{2}~dx. \label{E}%
\end{equation}
It is not difficult to compute that the variation of the energy $\frac{\delta
E_{\mu}(u)}{\delta u}:\mathcal{X}\rightarrow L^{2}(\Omega;\mathbb{R}^{4})$ is
given by
\begin{align*}
\frac{\delta E_{\mu}(u)}{\delta u}  &  =\left(  \frac{\delta E_{\mu}}{\delta
u_{1}},...,\frac{\delta E_{\mu}}{\delta u_{4}}\right)  ,\\
\frac{\delta E_{\mu}(u)}{\delta u_{j}}  &  =-\Delta u_{j}-\mu u_{j}+gu_{j}
^{3}+\tilde{g}\sum_{i=1(i\neq j)}^{4}u_{i}^{2}u_{j},
\end{align*}
where
\begin{equation}
\mathcal{X}=\left\{  u\in H^{2}(\Omega;\mathbb{R}^{4}):\left.  \frac{\partial
u}{\partial n}\right\vert _{\partial\Omega}=0\right\}  . \label{Space}%
\end{equation}

To find the gradient, we define the unbounded operator
\begin{equation}
L:=-\Delta+I:\mathcal{X}\subset L^{2}\rightarrow L^{2}.
\end{equation}
Since $L$ is a closed self-adjoint Fredholm operator of index zero, then
$L^{-1}:L^{2}\rightarrow\mathcal{X}\subset L^{2}$ exists and is completely
continuous. Moreover, the graph norm related to $L$ is equivalent to the
standard Sobolev norm $\mathcal{X}$, i.e. $\left\langle w,v\right\rangle
_{\mathcal{X}}=\left\langle Lw,v\right\rangle _{L^{2}}$. Therefore, we have
that $L^{-1}$ is the Riesz map, and we can conclude that the gradient of the
energy is given by
\begin{equation}
\nabla_{u}E_{\mu}(u)=L^{-1}\frac{\delta E_{\mu}(u)}{\delta u}\text{,}%
\end{equation}
which follows from the fact that $dE_{\mu}(u)v=\left\langle \frac{\delta
E_{\mu}(u)}{\delta u},v\right\rangle _{L^{2}}=\left\langle L^{-1}\frac{\delta
E_{\mu}(u)}{\delta u},v\right\rangle _{\mathcal{X}}$.

With a further computation we can arrive at the expression for the gradient
\begin{equation}\label{eq:gE}
\nabla_{u}E_{\mu}(u)=u+L^{-1}K_{\mu}(u), %
\end{equation}
where
\begin{align*}
K_{\mu}(u)  &  =\left(  K_{1}(\mu,u),...,K_{4}(\mu,u)\right)  ,\\
K_{j}(\mu,u)  &  =-(\mu+1)u_{j}+gu_{j}^{3}+\tilde{g}\sum_{i=1(i\neq j)}
^{4}u_{i}^{2}u_{j}.
\end{align*}
Notice that the operator $K_{\mu}:\mathcal{X}\rightarrow\mathcal{X}$ is
a well-defined and continuous operator. Indeed, since $H^{2}(\Omega,{\mathbb{R}
}^{4})$ is a Banach algebra, then also $\mathcal{X}$ is a Banach algebra. The
facts that $\mathcal{X}$ is a Banach algebra and $K_{\mu}$ is polynomial in
the components of $u$ imply that $K_{\mu}(u)\in\mathcal{X}$ if $u\in
\mathcal{X}$ . Moreover, the fact that $L^{-1}:\mathcal{X}\rightarrow
\mathcal{X}$ is a compact operator implies, by Sobolev embedding theorems, that
$L^{-1}K_{\mu}(u)$ is a compact operator.

Notice that the energy $E_{\mu}(u)$ defined in (\ref{E}) is $G$-invariant and
the gradient $\nabla_{u}E_{\mu}(u)$ given in \eqref{eq:gE} is $G$-equivariant
operator. The $G$-invariance of $E_{\mu}$ follows directly from the fact that the
four components are symmetric under permutation elements in $S_{4}$.

The fact that the constant function $u_{\mu}\in\mathcal{X}$ defined in
(\ref{eq:c}) is a critical solution of $\nabla_{u}E_{\mu}(u_{\mu})=0$ follows
from the computation
\[
\frac{\delta E_{\mu}(u_{\mu})}{\delta u_{j}}=\left(  -\mu+(g+3\tilde{g}%
)c_{\mu}^{2}\right)  c_{\mu}=0\text{.}%
\]
Thus the steady-state solution $u_{\mu}$ is the \textit{trivial solution} to
\eqref{GPE}. Since the isotropy group of $u_{\mu}$ is $G_{u_{\mu}}=O(3)\times
S_{4}$, then the orbit of $u_{\mu}$ is trivial.

We denote the set of trivial solutions by $\mathcal{T}$, i.e.
\[
\mathcal{T}=\{(u_{\mu}, \mu); \mu\in\mathbb{R}^{+}\}\subset\mathcal{X}
\times\mathbb{R}^{+} .
\]
Our aim is to consider the global bifurcation of solutions of \eqref{GPE} from
this family. Below we formulate the definition and some facts concerning the
global bifurcation.

Put
\[
\mathcal{N}=\{(u,\mu) \in\mathcal{X} \times\mathbb{R}^{+} \setminus
\mathcal{T}; \nabla_{u}E_{\mu}(u)=0\}.
\]

\begin{definition}
We say that a global bifurcation from $(u_{\mu_{0}}, \mu_{0})\subset
\mathcal{T}$ occurs if there exists a connected component $\mathcal{C}(\mu
_{0}) $ of $\overline{\mathcal{N}}$, containing $(u_{\mu_{0}}, \mu_{0})$ and such
that either $\mathcal{C}(\mu_{0}) $ is unbounded or it is bounded and
$\mathcal{C}(\mu_{0})\cap(\mathcal{T} \setminus\{(u_{\mu_{0}}, \mu_{0})\})
\neq\emptyset. $
\end{definition}

From the implicit function theorem the following necessary
condition for global bifurcation immediately follows: if $(u_{\mu_{0}}, \mu_{0})\subset\mathcal{T}$
is a point of global bifurcation, then the Hessian $\nabla^{2}E_{\mu_0}(u_{\mu_0})$ is
not an isomorphism. Denote by $\Lambda$ the set of $\mu$ satisfying this condition.

To formulate the sufficient condition we use the theory of $G$-equivariant
gradient degree. Assume that there exist $\mu_{0} \in\Lambda$ and $\mu_{-}
<\mu_{0} <\mu_{+}$ such that $[\mu_{-},\mu_{+}]\cap\Lambda= \{\mu_{0}\}.$ Then
the degrees $\nabla_{G}\text{-deg}(\nabla_{u}E_{\mu_{\pm}}(u), \mathcal{U})$
are well defined for $\mathcal{U}\subset\mathcal{X}$ being the sufficiently
small neighbourhood of $u_{\mu}$.
Therefore we can define the bifurcation index (the equivariant topological
invariant) by the formula
\begin{equation}
\omega_{G}(\mu_{0}):=\nabla_{G}\text{-deg\thinspace}(\nabla_{u}E_{\mu_{-}%
}(u),\mathcal{U})-\nabla_{G}\text{-deg\thinspace}(\nabla_{u}E_{\mu_{+}%
}(u),\mathcal{U}). \label{eq:bifind}%
\end{equation}
The nontriviality of this index implies the global bifurcation. More
precisely, there holds the following theorem:

\begin{theorem}
\label{thm:glob1} If $\omega_{G}(\mu_{0}) \neq0 \in U(G)$, then a global
bifurcation from $(u_{\mu_{0}}, \mu_{0})$ occurs. Moreover, for every non-zero
coefficient $m_{j}$ in
\[
\omega_{G}(\mu_{0})=m_{1}(K_{1})+m_{2}(K_{2})+\dots m_{r}(K_{r}),
\]
there exists a global family of non-trivial solutions with symmetries at least
$K_{j}$.
\end{theorem}

The proof of this theorem follows the same idea as in the case of the
Leray-Schauder degree, see Chapter 3, Section 4 of \cite{Ni}.

From the definition of $\mu_{\pm}$ and the homotopy property of the degree it follows, that the degrees in
\eqref{eq:bifind} can be computed by the linearization, i.e.
\begin{equation}
\label{eq:lin}\nabla_{G}\text{-deg\thinspace}(\nabla_{u}E_{\mu_{\pm}%
}(u),\mathcal U)=\nabla_{G}\text{-deg\thinspace}(\nabla^{2}_{u}E_{\mu_{\pm}%
}(u_{\mu_{\pm}}),B(\mathcal{X})).
\end{equation}
On the other hand, the degree of the linear map is the product of the
$G$-equivariant basic degrees (see Appendix \ref{app:gradient}, formula \eqref{eq:grad-lin}, for the definition of basic degrees). Hence, to
compute $\omega_{G}(\mu_{0})$ we need to determine the isotypic decomposition
of negative eigenspaces of $\nabla^{2}_{u}E_{\mu_{\pm}}(u_{\mu_{\pm}})$.

\section{Spectrum and isotypic decomposition for the Hessian}

In this section we describe the spectrum and the eigenspaces of the Hessian of
the energy $\nabla_{u}^{2}E_{\mu}(u)$. We start the analysis with the spectral
properties of the Laplace operator.

\subsection{Laplace operator}

Since the Hessian involves a Laplace operator, we start with the study of
properties of the spectrum of $-\Delta:X\rightarrow L^{2}(\Omega;\mathbb{R)}$
with Neumann boundary conditions, where
\[
X=\left\{  u\in H^{2}(\Omega;\mathbb{R)}:\left.  \frac{\partial u}{\partial
n}\right\vert _{\partial\Omega}=0\right\}  .
\]
One can use the well known method of separation of variables using the
spherical coordinates (see Appendix \ref{app:laplace}), to obtain the formula
for the spectrum of the Laplace operator
\[
\sigma(-\Delta)=\left\{  s_{km}^{2}:(k,m)\in \mathcal K\right\}  ,
\]
where $\mathcal K:=\{(0,0)\}\cup \{(k,m)\in \mathbb N\times \mathbb N^+\}, s_{00}=0$ and in other cases, %for $(k,m)\neq(0,0)$, the eigenvalue
$s_{km}^{2}\ $corresponds to the $m$-th positive zero of the function
\begin{equation}
\psi_{k}(\lambda):=\left.  \frac{d}{dr}(r^{-\frac{1}{2}}J_{k+\frac{1}{2}
}(\lambda r))\right\vert _{r=1}, \label{ceros}%
\end{equation}

Here, $J_{k+(1/2)}$ represents the $(k+(1/2))$-th Bessel function of the first
kind. To each eigenvalue $s_{km}^{2}$ there is associated an irreducible $SO(3)$
-representation with dimension $2k+1$ given by: $\mathcal{E}
_{00}=\left\langle 1\right\rangle $ for the eigenvalue $s_{00}^{2}=0$, and
\[
\mathcal{E}_{km}=\left\langle r^{-\frac{1}{2}}J_{k+\frac{1}{2}}(s_{km}
r)P_{k}^{n}
(\cos\varphi)\Big(a\cos(n\theta)+b\sin(n\theta)\Big):a,b\in \mathbb R,0\leq n\leq k\right\rangle ,
\]
for the eigenvalue $s_{km}^{2}$ with $(k,m)\neq(0,0)$, where
$P_{k}^{n}$ is the \textit{Legendre function} (see Appendix \ref{app:laplace} for the derivation).

The irreducible $SO(3)$-representations can be described using homogeneous
polynomials (for more details we refer to \cite{GSS}). We denote by $W_{k}$,
$k=0,1,2,3,\dots$ and $W_{k}:=\{0\}\;$for $\;k<0$, the space of homogeneous
polynomials $p:\mathbb{R}^{3}\rightarrow\mathbb{R}$ of degree $k$. Clearly
$\mathfrak{g}\in SO(3)$ acts on $W_{k}$ by
\[
(\mathfrak{g}p)(v)=p\left(  \mathfrak{g}^{-1}v\right)  ,\quad v=(x,y,z)^{T}\in\mathbb{R}^{3}.
\]
In addition, define $\rho:\mathbb{R}^{3}\rightarrow\mathbb{R}$ by
$\rho(x,y,z)=x^{2}+y^{2}+z^{2}$, and put
\[
U_{k}:=\{\rho p:p\in W_{k-2}\}.
\]
Then, the space $W_{k}$, $k=0,1,2,\dots$, admits an $SO(3)$-invariant direct
sum decomposition
\[
W_{k}=U_{k}\oplus\mathcal{V}_{k}.
\]
It is known that the $O(3)$-representations $\mathcal{V}_{k}$, $k=0,1,2,\dots
$, are absolutely irreducible such that $\dim\,\mathcal{V}_{k}=1+2k$ and the
space $\mathcal{E}_{km}$ is $O(3)$-equivalent to $\mathcal{V}_{k}$.

Since $L=-\Delta+I:\mathcal{X}\rightarrow\mathcal{X}$ with $\mathcal{X}=X^{4}
$, we have that
\[
\sigma(L)=\left\{  s_{km}^{2}+1:(k,m)\in\mathcal K\right\}  ,
\]
where the eigenvalue $s_{km}^{2}+1$ has eigenspace $\left(  \mathcal{E}
_{km}\right)  ^{4}< X^{4}=:\mathcal{X}$ with dimension $4(2k+1)$. Set
\[
\mathcal{X}_{k}:=\overline{\bigoplus_{m=1}^{\infty}\left(  \mathcal{E}
_{km}\right)  ^{4}}.
\]
where the closure is taken in $\mathcal{X}$. Then the $O(3)$-isotypic
decomposition of $\mathcal{X}$ is
\[
\mathcal{X}=\overline{\bigoplus_{k=0}^{\infty}\mathcal{X}_{k}}~,
\]
where $\mathcal{X}_{k}$ is the $O(3)$-isotypic component modeled on the
irreducible representation of the spherical harmonic $\mathcal{V}_{k}$.

\subsection{Spectrum of the Hessian}

Now we can compute the Hessian of the energy.

\begin{proposition}
The Hessian $\nabla^{2}E_{\mu}(u_{\mu}):\mathcal{X}\rightarrow\mathcal{X}$ is the
linear operator
\[
\nabla^{2}E_{\mu}(u_{\mu})=L^{-1}\left(  -\Delta I+2c_{\mu}^{2}M\right)  ,
\]
where $M$ is
\[
M=\left(
\begin{array}
[c]{cccc}%
g & \tilde{g} & \tilde{g} & \tilde{g}\\
\tilde{g} & g & \tilde{g} & \tilde{g}\\
\tilde{g} & \tilde{g} & g & \tilde{g}\\
\tilde{g} & \tilde{g} & \tilde{g} & g
\end{array}
\right)  .
\]

\end{proposition}

\begin{proof}
Since the variation of energy satisfies
\[
\frac{\delta}{\delta u_{j}}E_{\mu}(u)=-\Delta u_{j}-\mu u_{j}+gu_{j}
^{3}+\tilde{g}\sum_{i=1(i\neq j)}^{4}u_{i}^{2}u_{j}\text{,}
\]
we have for $i\neq j$ that $\frac{\delta^{2}}{\delta u_{i}\delta u_{j}}E_{\mu
}(u)=2\tilde{g}u_{i}u_{j}$, and that
\[
\frac{\delta^{2}}{\delta u_{j}\delta u_{j}}E_{\mu}(u)=-\Delta-\mu+3gu_{j}
^{2}+\tilde{g}\sum_{i=1(i\neq j)}^{4}u_{i}^{2}.
\]
Since $-\mu\ +3\tilde{g}c_{\mu}^{2}=-gc_{\mu}^{2}$, then the components
satisfy
\[
\frac{\delta^{2}}{\delta u_{i}\delta u_{j}}E_{\mu}(u_{\mu})=2\tilde{g}c_{\mu
}^{2},\qquad\frac{\delta^{2}}{\delta u_{j}\delta u_{j}}E_{\mu}(u_{\mu
})=-\Delta+2gc_{\mu}^{2}\text{.}
\]
The result follows by writing in vectorial form that $\frac{\delta^{2}}{\delta
u^{2}}E_{\mu}(u)=-\Delta I+2c_{\mu}^{2}M$ and using that $\nabla^{2}E_{\mu
}(u_{\mu})=L^{-1}\frac{\delta^{2}}{\delta u^{2}}E_{\mu}(u)$.
\end{proof}

Define $\{e_{j}\}_{j=1}^{4}$ as the natural basis of $\mathbb{R}^{4}$.
Therefore, the spectrum of $\nabla^{2}E_{\mu}(u_{\mu})$ is determined by the
eigenvalues of the matrix $M$. We claim that the eigenvalues of the matrix $M$ are:

\begin{itemize}
\item $g-\tilde{g}$ with multiplicity $3$. The matrix has the eigenvalue
$g-\tilde{g}$ with multiplicity $3$ for the eigenvectors $e_{1}-e_{2}%
,e_{1}-e_{3},e_{1}-e_{4}$.

\item $g+3\tilde{g}$ with multiplicity $1$. The matrix has the eigenvalue
$g+3\tilde{g}$ with eigenvector $e_{1}+e_{2}+e_{3}+e_{4}$.
\end{itemize}

For a function $u_{km}\in(\mathcal{E}_{km})^{4}$ we have
\[
\nabla^{2}E_{\mu}(u_{\mu})u_{km}=L^{-1}\left(  -\Delta I+2c_{\mu}^{2}M\right)
u_{km}=\frac{s_{km}^{2}I+2c_{\mu}^{2}M}{s_{km}^{2}+1}u_{km}\text{.}
\]
Therefore, the space $(\mathcal{E}_{km})^{4}$ with dimension $4(2k+1)$ is
invariant under the application of the Hessian operator $\nabla^{2}E_{\mu
}(u_{\mu}).$ We conclude that the spectrum of the Hessian consists of
\begin{equation}
\label{eq:spectrum}\sigma(\nabla^{2}E_{\mu}(u_{\mu}))=\left\{
\begin{array}
[c]{c}%
\frac{s_{km}^{2}+2c_{\mu}^{2}\left(  g-\tilde{g}\right)  }{s_{km}^{2}+1}\text{
mult }3(2k+1)\\
\frac{s_{km}^{2}+2c_{\mu}^{2}\left(  g+3\tilde{g}\right)  }{s_{km}^{2}
+1}\text{ mult }(2k+1)
\end{array}
:(k,m)\in{\mathcal K}\right\}  .
\end{equation}

We mostly consider the immiscible case $0<g<\tilde{g}$. Therefore the
eigenvalues
\[
\xi_{km}:=\frac{s_{km}^{2}+2c_{\mu}^{2}\left(  g-\tilde{g}\right)  }
{s_{km}^{2}+1},\quad(k,m)\in \mathcal K\backslash\{(0,0)\},
\]
are the only ones that cross zero at the critical value $\mu_{km}$, which is obtained by setting $\xi_{km}=0$ and using the relation $c_\mu^2 = \mu/(g+3\tilde{g})$ from \eqref{eq:c} (with $n=4$). This yields
\begin{equation}
\mu_{km}:=\frac{g+3\tilde{g}}{2\left(  \tilde{g}-g\right)  }s_{km}^{2}
,\quad(k,m)\in \mathcal K\backslash\{(0,0)\}\text{.} \label{eq:mu_km}%
\end{equation}
Since all the other eigenvalues are positive, there will be no bifurcation at
$\mu_{km}$ if $\mu$ is not a critical value. Therefore, we have the following
critical set associated with \eqref{E}:
\begin{equation}
\Lambda:=\big\{\mu_{km}\in\mathbb{R}:(k,m)\in \mathcal K\backslash\{(0,0)\}\big\}.
\end{equation}

\subsection{Isotypic decomposition}

\label{sec:isotypic}

We have that the $O(3)$-irreducible representations $\mathcal{E}_{km}$ are
equivalent to the representation $\mathcal{V}_{k}$ for $k\in\mathbb{N}$ with
dimension $2k+1$. Furthermore, the action of $O(3)\ $in $(u_{1},u_{2}
,u_{3},u_{4})\in\left(  \mathcal{E}_{km}\right)  ^{4}$ is identical in each
component of $\mathcal{E}_{km}$, while $S_{4}$ permutes the components $u_{j}$
for $j=1,...,4$. Therefore, we have $(\mathcal{E}_{km})^{4}=\mathbb{R}
^{4}\otimes\mathcal{E}_{km}$, where the group $S_{4}$ acts in each
$\mathbb{R}^{4}$ by permutation of indices.

The action of $S_{4}$ has the following characters for irreducible representations

\begin{center}
\label{eq:table}
\begin{tabular}
[c]{|c|c|c|c|c|c|c|}\hline
Rep. & Character & (1) & (1,2) & (1,2)(3,4) & (1,2,3) & (1,2,3,4)\\\hline
$\mathcal{W}_{0}$ & $\chi_{0}$ & 1 & 1 & 1 & 1 & 1\\
$\mathcal{W}_{1}$ & $\chi_{1}$ & 1 & -1 & 1 & 1 & -1\\
$\mathcal{W}_{2}$ & $\chi_{2}$ & 2 & 0 & 2 & -1 & 0\\
$\mathcal{W}_{3}$ & $\chi_{3}$ & 3 & -1 & -1 & 0 & 1\\
$\mathcal{W}_{4}$ & $\chi_{4}$ & 3 & 1 & -1 & 0 & -1\\\hline
$\mathbb{R}^{4}$ & $\chi_{\mathbb{R}^{4}}$ & $4$ & $2$ & $0$ & $1$ &
$0$\\\hline
\end{tabular}

\end{center}

where $\chi_{j}$ is the character of the irreducible representation
$\mathcal{W}_{j}$ of the group $S_{4}.$ Using the character table one can
conclude that the space $\mathbb{R}^{4}$ decomposes in the $S_{4}$-irreducible
representations $\mathcal{W}_{0}$ and $\mathcal{W}_{4}$, i.e. $\mathbb{R}
^{4}=\mathcal{W}_{0}\oplus\mathcal{W}_{4}$. We conclude that
\[
(\mathcal{E}_{km})^{4}=(\mathcal{E}_{km}\otimes\mathcal{W}_{0})\oplus
(\mathcal{E}_{km}\otimes\mathcal{W}_{4})\text{.}
\]

Therefore, we have the following $O(3)\times S_{4}$ isotypic components of the
space $\mathcal{X}:$%

\[
\mathcal{X}=\overline{\bigoplus_{k=0}^{\infty}\left(  \mathcal{X}_{k}
^{0}\oplus\mathcal{X}_{k}^{4}\right)  },
\]
where
\[
\mathcal{X}_{k}^{0}=\overline{\bigoplus_{m=1}^{\infty}\mathcal{E}_{km}
\otimes\mathcal{W}_{0}}~,\quad\mathcal{X}_{k}^{4}=\overline{\bigoplus
_{m=1}^{\infty}\mathcal{E}_{km}\otimes\mathcal{W}_{4}}~.
\]
The isotypic component $\mathcal{X}_{k}^{0}$ is modeled on the $O(3)\times
S_{4}$-irreducible representation $\mathcal{V}_{k}\otimes\mathcal{W}_{0}$ and
$\mathcal{X}_{k}^{4}$ in $\mathcal{V}_{k}\otimes\mathcal{W}_{4}$. The
eigenspace of the eigenvalue $\xi_{km}$, with multiplicity $3(2k+1),$ is
$\mathcal{V}_{k}\otimes\mathcal{W}_{4}$, where the action of $(\mathfrak{g}_1,\mathfrak{g}_2)\in
O(3)\times S_{4}$ on $v\otimes w\in\mathcal{V}_{k}\otimes\mathcal{W} _{4}$ is
given by
\[
(\mathfrak{g}_1,\mathfrak{g}_2)\cdot v\otimes w=(\mathfrak{g}_1 v)\otimes(\mathfrak{g}_2 w)\text{.}
\]
\section{Implementation of the gradient degree}

We will use the following theorem to prove our Theorem \ref{th:main1}.

\begin{theorem}
\label{Thm2}Under \textit{immiscible }assumption $g<\tilde{g}$, the operator
$\nabla_{u}E_{\mu}(u)$ has $G$-orbit of branches bifurcating at critical value
$\mu_{km}$ with symmetries at least $(H_{j})$ for each orbit type $(H_{j})\in\max(b)$
(see \eqref{eq:max-in-a}), where
\begin{equation}
b:=(G)-\prod_{\left\{  (j,n)\in\mathcal K:s_{jn}=s_{km}\right\}  }%
\nabla_{G}\text{-}\deg_{\mathcal{V}_{j}\otimes\mathcal{W}_{4}}=\sum_{j}\alpha
_{j}(H_{j})+\bm\beta. \label{deg2}%
\end{equation}

\end{theorem}

\begin{proof}
To prove that the global bifurcation occurs, we use Theorem \ref{thm:glob1}.
Therefore we consider the bifurcation index $\omega_{G}(\mu_{km})$. Taking
into account formula \eqref{eq:lin} we obtain that
\[
\omega_{G}(\mu_{km}):=\nabla_{G}\text{-deg}\,(\nabla_{u}^{2}E_{\mu_{-}%
}(u_{\mu_-}),B(\mathcal X))-\nabla_{G}\text{-deg}\,(\nabla_{u}^{2}E_{\mu_{+}%
}(u_{\mu_+}),B(\mathcal X)).
\]
To compute the above degrees we use the formula \eqref{eq:grad-lin}.
Observe
that from \eqref{eq:spectrum} it follows that $\sigma_{-}(\nabla^{2}E_{\mu
}(u_{\mu}))=\bigcup_{\left\{  (j,n)\in\mathcal K:\mu_{jn}<\mu\right\}
}\{\xi_{jn}\}$ and, as observed in Section \ref{sec:isotypic}, we consider
that the eigenvalue $\xi_{km}$ has the eigenspace $\mathcal{E}_{km}%
\otimes\mathcal{W}_{4}$, which is isomorphic to the $G$-irreducible
representation $\mathcal{V}_{k}\otimes\mathcal{W}_{4}.$ Hence, we obtain%
\begin{align*}
\omega_{G}(\mu_{km})  &  =\prod_{\left\{  (j,n)\in\mathcal K:s_{jn}%
<s_{km}\right\}  }\nabla_{G}\,\text{-deg}_{\mathcal{V}_{j}\otimes
\mathcal{W}_{4}}-\prod_{\left\{  (j,n)\in\mathcal K:s_{jn}\leq
s_{km}\right\}  }\nabla_{G}\,\text{-deg}_{\mathcal{V}_{j}\otimes
\mathcal{W}_{4}}\\
&  =\left(  \prod_{\left\{  (j,n)\in \mathcal K:s_{jn}<s_{km}\right\}
}\nabla_{G}\,\text{-deg}_{\mathcal{V}_{j}\otimes\mathcal{W}_{4}}\right)
\left(  (G)-\prod_{\left\{  (j,n)\in\mathcal K:s_{jn}=s_{km}\right\}
}\nabla_{G}\text{-deg}_{\mathcal{V}_{j}\otimes\mathcal{W}_{4}}\right) \\
&  =:a\ast b.
\end{align*}

Proposition \ref{prop:invertible} implies that the first factor $a$ is
invertible, therefore the nontriviality of $\omega_{G}(\mu_{km})$ is
equivalent to the nontriviality of the latter factor $b$. Furthermore, since
$(H_{j})$ is maximal in $\Phi(b)$ (see \eqref{eq:Phi(a)}), it follows (by
Corollary \ref{cor:inv}) that
\[
\text{coeff}^{H_{j}}(\omega_{G}(\mu_{km}))=\text{coeff}^{H_{j}}(a\ast
b)\not =0.
\]
Therefore, we conclude that $\nabla_{u}E_{\mu}(u)$ has a $G$-orbit of branches
bifurcating at critical value $\mu_{km}$ with maximal symmetries $(H_{j})$ for
each subgroup $(H_{j})\in \max(b)$. 
\end{proof}

\subsection{Bifurcation in the general case}

In order to compute the maximal isotropy groups in $\nabla_{G}$
-deg$_{\mathcal{V}_{k}\otimes\mathcal{W}_{4}}$ we need to describe the
conjugacy classes of subgroups of $G=O(3)\times S_4$. We refer to Chapter XIII of \cite{GSS} for notation of subgroups of $O(3)$,  to Section 5 of \cite{AED} for subgroups of $S_4.$ In this work, we adopt the amalgamated notation
for subgroups of $G=O(3)\times S_4.$ A detailed explanation of amalgamated subgroups in product group $G_1\times G_2$ can be found in Appendix E of \cite{newBook}.

\begin{theorem}
\label{Thm3}Under \textit{immiscible }assumption $g<\tilde{g}$, the operator
$\nabla_{u}E_{\mu}(u)$ has a global bifurcation of $G$-orbits of non-trivial
solutions bifurcating from $u_{\mu}$ at each critical value $\mu_{km}$ with
$(k,m)\in \mathcal K\backslash\{(0,0)\}$. Furthermore, in the case $k=0$,
there are two global branches with symmetries at least $(O(3)\times
D_{2})$ and $(O(3)\times D_{3})$.
\end{theorem}

\begin{proof}
By the previous theorem we only need to prove that
\begin{equation}
b:=(G)-\prod_{\left\{  (j,n)\in\mathcal K:s_{jn}=s_{km}\right\}  }%
\nabla_{G}\text{-deg}_{\mathcal{V}_{j}\otimes\mathcal{W}_{4}}\label{product}%
\end{equation}
contains an element. We start by noticing that the eigenvalues $s_{0n}^{2}$
are increasing in $n$ because this number corresponds to the $n$th positive
zero of the function (\ref{ceros}). Thus we can consider two cases: (i) there
is only one $n_{0}\in\mathbb{N}^{+}$ such that
\[
(0,n_{0})\in\left\{  (j,n)\in\mathcal K:s_{jn}=s_{km}\right\}
\]
or (ii) there is none.

If (ii), it is known that $\mathcal{V}_{j}$ is a nontrivial $O(3)$%
-representation for all $\{(j,n)\in\mathbb{N}^{2}:s_{jn}=s_{km}\}$. Hence,
since $S^{1}\subset G$ it follows that $\mathcal{V}_{k}\otimes\mathcal{W}_{4}$
is a nontrivial $S^{1}$-representation, therefore from Corollary 3.3 of
\cite{GRH} and the functoriality property, we have that
%\begin{equation}
%b=\alpha(H)+...\tex%t{,}%
%\end{equation}
%where $\alpha \neq 0$.
b is a non-zero element of $U(G)$ and we obtain the nontriviality of the bifurcation index.

In case (i), we have only one eigenvalue $s_{0n}^{2}$ corresponding to the
representation $\mathcal{V}_{0}\otimes\mathcal{W}_{4}=\mathcal{W}_{4}$ with
$\mathcal{V}_{0}$ the trivial $O(3)$-irreducible representation of dimension
$1$.
Notice that the basic gradient degree $\nabla_{G}\text{-deg}%
_{\mathcal{W}_{4}}$ coincides with the basic Brouwer $G$-equivariant degree $\text{deg}
_{\mathcal{W}_{4}}$, thus we have:
$(\nabla_{G}\,\text{-deg}_{\mathcal{W}_{4}})^2=(G)$, and see (\cite{AED}, chapter 5):
\[
\nabla_{G}\text{-deg}_{\mathcal{W}_{4}}=(G)-(O(3)\times D_{2})-2(O(3)\times
D_{3})+3(O(3)\times D_{1})-(O(3)),
\]
Therefore, we have that%
\begin{align*}
b: & =\nabla_{G}\,\text{-deg}_{\mathcal{W}_{4}}\ast\left(  \nabla_{G}\text{-deg}%
_{\mathcal{W}_{4}}-\prod_{\left\{  (j,n)\in\mathcal K:s_{jn}%
=s_{km},~~j\neq0\right\}  }\nabla_{G}\text{-deg}_{\mathcal{V}_{j}%
\otimes\mathcal{W}_{4}}\right)  \\
&  =\nabla_{G}\,\text{-deg}_{\mathcal{W}_{4}}\ast c, \quad
\end{align*}
where $$ c =-(O(3)\times D_{2})-2(O(3)\times
D_{3})+3(O(3)\times D_{1})-(O(3))+\sum_{l}\alpha_l(H_l),
$$
and $\dim H_l=0$ or 1 (see \cite{GSS}, Theorem 9.6). Thus,
$(O(3)\times D_{2}),\;(O(3)\times D_{3})\in\max(c)$
and the nontriviality of the bifurcation index follows from Corollary \ref{cor:inv}.
Furthermore, in the case $k=0$, since $s_{km}\not=s_{00}$, locally the solutions in the bifurcating branches are not constant.
Therefore we obtain the existence of two global branches with symmetries at least $(O(3)\times D_{2})$ and $(O(3)\times D_{3})$.
\end{proof}
\vskip 0.3cm
The fact that the global branches for $k=0$ have global
symmetries at least $(O(3)\times D_{2})$ and $(O(3)\times D_{3})$ implies that the
solutions are radial solutions. Moreover, in the solution with symmetries $(O(3)\times
D_{3})$ three of the four components are equal. In the solutions with
symmetries $(O(3)\times D_{2})$ we have two different pairs of equal components.
Since we do not know if $(O(3)\times D_{2}),\;(O(3)\times D_{3})\in\max(b)$,
we can only guarantee that locally the solutions  are
not constant because the only constant solutions can bifurcate from $\mu_{00}$ and $s_{km}\not=s_{00}$.

\subsection{Bifurcation in the case $\mathcal{V}_{1}$}\label{sec:bifcase}

The goal of this section is to show the full power of the gradient degree,
finding the maximal orbit types in $\Phi_{0}(G;\mathcal{V}_{1}\otimes
\mathcal{W}_{4}\setminus\{0\})$ with finite Weyl groups, for which the
corresponding coefficient of the basic degree deg$_{\mathcal{V}_{1}%
\otimes\mathcal{W}_{4}}$ are non-zero. % presented in Appendix \ref{app:subgroups}.

Notice that numerical evidence indicates that the values $s^2_{km}$ do not repeat but there is no formal proof of this fact available. Therefore, in order to avoid a possible mistake, we assume that the considered eigenvalues $s^2_{km}$ are $O(3)$-isotypic simple, and consequently $\mu_{km}$ are $O(3)\times S_4$ isotypic simple.

\begin{theorem}
\label{Thm4}Under \textit{immiscible }assumption $g<\tilde{g}$, the operator
$\nabla_{u}E_{\mu}(u)$ has $6$ global branches of $G$-orbit bifurcating from
$u_{\mu}$ at the isotypic simple critical values $\mu_{1m}$ for $m\in\mathbb{N}^{+}$.
\end{theorem}

\begin{proof}
For this representation we take advantage of the fact that the restriction of
the action of the group $S_{4}<O(3)$ to the representation $\mathcal{V}_{1}$
is exactly $\mathcal{W}_{3}$ (natural representation). Consequently, the
subgroup
\[
G':=S_{4}^{p}\times S_{4}\leq O(3)\times S_{4},\quad S_{4}^{p}:=S_{4}%
\times{\mathbb{Z}}_{2},
\]
acts on the space $\mathcal{V}_{1}\otimes\mathcal{W}_{4}$, where $S_{4}\leq
O(3)$ acts naturally on $\mathcal{V}_{1}$ as $\mathcal{W}_{3}$, i.e.
$\mathcal{V}_{1}$ as a representation of $S_{4}\times{\mathbb{Z}}_{2}$ is
$\mathcal{W}_{3}^{-}$ (with the antipodal ${\mathbb{Z}}_{2}$-action).
Consequently, as a $G'$-representation, $\mathcal{V}_{1}\otimes
\mathcal{W}_{4}$ is equivalent to the irreducible $S_{4}^{p}\times S_{4}%
$-representation $\mathcal{W}_{3}^{-}\otimes\mathcal{W}_{4}$.
We use GAP to compute the basic degree. The code is uploaded to github:
https://github.com/violal1016/Bose-Einstein-Condensate. By using the GAP
system one can compute the corresponding basic $G'$-equivariant degree
\begin{align}
\nabla_{G'}\text{-deg}_{\mathcal{W}_{3}^{-}\otimes\mathcal{W}_{4}}  &
=(G')-({D_{4}^{p}}^{^{{\mathbb{Z}}_{2}^{-}}}\times{_{D_{4}}}D_{4}%
)_{1}-({D_{4}^{p}}^{^{{\mathbb{Z}}_{2}^{-}}}\times{_{D_{4}}}D_{4})_{2}%
-({D_{4}^{p}}^{^{D_{4}^{z}}}\times{^{D_{2}}}D_{4})\label{eq:g1}\\
&  -({D_{4}^{p}}^{^{D_{4}^{z}}}\times{^{D_{1}}}D_{2})-({D_{3}^{p}}%
^{^{D_{3}^{z}}}\times{^{D_{1}}}D_{2})-({D_{3}^{p}}^{^{D_{3}^{z}}}%
\times{^{D_{2}}}D_{4})\nonumber\\
&  -({D_{2}^{p}}^{^{D_{2}^{d}}}\times{^{D_{2}}}D_{4})-({D_{2}^{p}}^{{D_{2}%
^{d}}}\times{^{D_{1}}}D_{2})-(D_{4}^{z}\times D_{3})-(S_{4}^{-}\times_{S_{4}%
}S_{4}) \nonumber\\
&  -({D_{2}^{d}}\times D_{3})-({D_{3}^{z}}\times D_{3})-({D_{3}}\times
{_{D_{3}}}D_{3})+\bm\alpha,\nonumber
\end{align}
where $\bm\alpha$ denotes the element in the Euler ring $U(G')$ with all
the coefficients corresponding to sub-maximal orbit types. To give an idea
about the complexity of this degree, the lattice of the $G'$-orbit types in
$\mathcal{W}_{3}^{-}\otimes\mathcal{W}_{4}$ has a couple of hundred
subgroups, which are computed using GAP. According to the amalgamated notation
described in \cite{newBook}, the subgroups $({D_{4}^{p}}^{^{{\mathbb{Z}}_{2}%
^{-}}}\times_{D_{4}}D_{4})_{1}$ and $({D_{4}^{p}}^{^{{\mathbb{Z}}_{2}^{-}}%
}\times_{D_{4}}D_{4})_{2}$ in \eqref{eq:g1} indicates that there are two
epimorphisms, ${\varphi}_{1}:D_{4}^{p}\rightarrow D_{4}$ and $\varphi
_{2}:D_{4}^{p}\rightarrow D_{4}$ which are not conjugate in $G'$ but
conjugate in $G$, and $\text{\textrm{Ker\thinspace}}\varphi_{1}%
=\text{\textrm{Ker\thinspace}}\varphi_{2}={\mathbb{Z}}_{2}^{-}~.$

Let us recall (see the Appendix \ref{app:gradient}) that for a subgroup
$G'$ of $G$, i.e. the natural embedding $\psi:G'\rightarrow G$ induces the ring homomorphism $\Psi:U(G)\rightarrow U(G')$. This allows to derive (see Example \ref{ex:basic_degree} for details) that
\begin{align}
\nabla_{G}\text{-deg}_{\mathcal{V}_{1}\otimes\mathcal{W}_{4}}  &
=(G)-({D_{4}^{p}}^{^{{\mathbb{Z}}_{2}^{-}}}\times{_{D_{4}}}D_{4})-(O(2)^-\times D_3))\label{basicdegree2}
-( {O(2)^p}^{D_2^d}\times ^{D_1} D_2)\\
& -({O(2)^p}^{O(2)^-}\times ^{D_2}D_4)%
-(S_{4}^{-}\times_{S_{4}%
}S_{4})-({D_{3}}\times
{_{D_{3}}}D_{3})+\bm\beta,\nonumber
\end{align}
where $\bm\beta$ denotes the element in the Euler ring $U(G)$ with all the
coefficients corresponding to sub-maximal orbit types.\vs

We have the following six maximal orbit types in $\mathcal V_1\otimes \mathcal W_4\setminus \{0\}$:
\begin{align*}
\mathfrak m=\Big\{&({D_{4}^{p}}^{^{{\mathbb{Z}}_{2}^{-}}}\times{_{D_{4}}}D_{4}),(O(2)^-\times D_3),( {O(2)^p}^{D_2^d}\times ^{D_1} D_2),\\ &({O(2)^p}^{O(2)^-}\times ^{D_2}D_4),
(S_{4}^{-}\times_{S_{4}}S_{4}),({D_{3}}\times
{_{D_{3}}}D_{3})\}.
\end{align*}
The six global branches have the symmetries at least $H\in\mathfrak m$ respectively.
Observe that $(O(2)^-\times D_3), ( {O(2)^p}^{D_2^d}\times ^{D_1} D_2), ({O(2)^p}^{O(2)^-}\times ^{D_2}D_4),(S_{4}^{-}\times_{S_{4}
}S_{4})$ are maximal in $\Phi(G) \setminus \{(G)\}.$ Due to the assumption that the critical value is isotypic simple,
 the solutions in the corresponding local branches have symmetries exactly $H$. It is important to remark that the branches do not need to keep the same symmetries globally, there can be
 symmetry-breaking bifurcations along the global continuum to higher groups of symmetries.  %$\mathcal X\setminus \{0\}$.

 If $(K)\ge ({D_{4}^{p}}^{^{{\mathbb{Z}}_{2}^{-}}}\times{_{D_{4}}}D_{4})$, then
\[
(K)=({D_{4k}^{p}}^{^{{\mathbb{Z}}_{2k}^{-}}}\times{_{D_{4}}}D_{4})
\]
Similarly, for $(D_3\times _{D_3}D_3),$ $ (K)=(D_{3k}^{\mathbb Z_k}\times _{D_3}D_3).$ Thus, if $(H_1),(H_2)\in \mathfrak m$ are different orbit types, then
for all $ (K_1),(K_2) $ being orbit types in $\mathcal X\setminus \{0\}$ such that $(K_1)\ge (H_1)$ and $ (K_2)\ge (H_2), $ one has $(K_1)\neq (K_2).$ That implies that the branches corresponding to these orbit types must be different.
\end{proof}\vs

From the above reasoning it follows that the solutions in bifurcating branches have their components related by different groups of rotations.
The most notable case is the group $(S_{4}^{-}\times_{S_{4}}S_{4})$, where the four components
are related by a tetrahedral symmetry of rotations. The emergence of such polyhedral (non-dihedral) symmetry families is a principal reason for studying four components: for three components only dihedral symmetry types arise (cf.\ \cite{MM2}), so the richer structure of four components is needed to reveal these solutions.
This group of rotations is not simply of the form $D_n$ such as in the other solutions found in $\mathfrak m$.
\section{Description of Symmetries}
\label{sec:visualization}

We illustrate three symmetry types considered in \eqref{basicdegree2} by
computing their fixed-point spaces in $\mathcal V_1\otimes \mathcal W_4$ explicitly.
For the numerical illustrations, we choose $g=1<\widetilde g=2$, and let $\mu_{11},$ $c_{\mu_{11}}$ be defined
by \eqref{eq:mu_km} and \eqref{eq:c}, respectively, with $s_{11}$ determined by \eqref{eq:spL}.
Numerically, $s_{11}\approx 2.0816$ is the first positive zero of $\lambda J_{3/2}'(\lambda)-\frac12 J_{3/2}(\lambda)$. Hence, by \eqref{eq:mu_km}, $\mu_{11}=\frac{g+3\widetilde g}{2(\widetilde g-g)}s_{11}^2=\frac72 s_{11}^2\approx 15.165$, and, by \eqref{eq:c} with $n=4$, $c_{\mu_{11}}^2=\frac{\mu_{11}}{g+3\widetilde g}=\frac{s_{11}^2}{2}$, so that $c_{\mu_{11}}\approx 1.4719$.
Through the $G$-equivariant isomorphism $\mathcal E_{11}\otimes\mathcal W_4\cong\mathcal V_1\otimes \mathcal W_4$, we represent each selected symmetry type by a nonzero function $\Phi_H\in(\mathcal E_{11}\otimes \mathcal W_4)^H$ and consider the first-order perturbation
\[
u_H(x)
=
u_{\mu_{11}}+\varepsilon\Phi_H(x)
=
c_{\mu_{11}}\mathbf 1+\varepsilon\Phi_H(x),
\qquad
x\in \Omega,
\]
where $\mathbf 1=(1,1,1,1)^T$ and $\varepsilon$ is a small real
parameter. Writing $x=(x_1,x_2,x_3)^T\in\Omega$ and $r=|x|$, we have
\[
\mathcal E_{11}
=
\left\{
r^{-3/2}J_{3/2}(s_{11}r)
(ax_1+bx_2+cx_3)
:\ a,b,c\in\mathbb R
\right\}.
\]
Indeed, \eqref{eq: Legendre} gives
\[
P_1^0(\cos\varphi)=\cos\varphi,
\;\;
P_1^1(\cos\varphi)=\sin\varphi,
\;\; \varphi\in[0,\pi].
\]
Thus, in spherical coordinates
$x_1=r\sin\varphi\cos\theta,\;\;
x_2=r\sin\varphi\sin\theta,\;\;
x_3=r\cos\varphi,$
the eigenspace takes the form
\[
\mathcal E_{11}
=
r^{-1/2}J_{3/2}(s_{11}r)
\operatorname{span}
\{\cos\varphi,\sin\varphi\cos\theta,\sin\varphi\sin\theta\}.
\]
Consequently, under the above $G$-equivariant identification, every element of $\mathcal E_{11}\otimes\mathcal W_4$ can be represented by a function of the form
\begin{equation}\label{eq:v1otimesw4}
r^{-3/2}J_{3/2}(s_{11}r)\,Ax,
\end{equation}
where
\[
A=
\begin{pmatrix}
a_1&b_1&c_1\\
a_2&b_2&c_2\\
a_3&b_3&c_3\\
a_4&b_4&c_4
\end{pmatrix}
\in\mathbb R^{4\times3},
\qquad
\mathbf1^TA=0.
\]
Here we use that $\mathbb R^4=\mathcal W_0\oplus\mathcal W_4$ with $\mathcal W_0=\operatorname{span}\{\mathbf 1\}$, so that $\mathcal W_4=\{w\in\mathbb R^4:\mathbf 1^Tw=0\}$; this is why the condition $\mathbf 1^TA=0$ appears. For $\sigma\in S_4$ let $P_\sigma$ be the $4\times4$ permutation matrix with $P_\sigma e_i=e_{\sigma(i)}$, so that $(P_\sigma w)_j=w_{\sigma^{-1}(j)}$. By \eqref{eq: actions}, for $(\gamma,\sigma)\in O(3)\times S_4$ and $u(x)=r^{-3/2}J_{3/2}(s_{11}r)\,Ax$ we have
\[
\big((\gamma,\sigma)u\big)(x)=P_\sigma\, u(\gamma^{-1}x)=r^{-3/2}J_{3/2}(s_{11}r)\,P_\sigma A\gamma^{T}x ,
\]
because $|\gamma^{-1}x|=r$ and $\gamma^{-1}=\gamma^T$. Hence $u$ is fixed by a subgroup $H$ if and only if $P_\sigma A\gamma^T=A$ for every generator $(\gamma,\sigma)$ of $H$. Denoting by $a_j^T$ the $j$-th row of $A$, this condition reads
\[
\gamma a_j=a_{\sigma(j)},\qquad j=1,\dots,4 .
\]
The fixed-point subspaces below are computed by solving these linear equations.
\vskip0.3cm

\noindent \textbf{Symmetry of type $(S_4^-\times_{S_4}S_4):$} Notice that
$H=S_4^-\times_{S_4}S_4$ is generated by
\[
\left\langle
\bigl(((12),-1),(12)\bigr),
\bigl(((1234),-1),(1234)\bigr)
\right\rangle.
\]
Here \(S_4^-\) admits an identification with a subgroup of \(O(3)\), under which the spatial components of these generators correspond to the matrices
\[
M_{((12),-1)}
=
\begin{pmatrix}
0&0&-1\\
0&1&0\\
-1&0&0
\end{pmatrix},
\qquad
M_{((1234),-1)}
=
\begin{pmatrix}
0&1&0\\
-1&0&0\\
0&0&-1
\end{pmatrix}.
\]
More precisely, the image of $S_4^-$ in $O(3)$ is the full symmetry group (rotations and reflections) of the regular tetrahedron with vertices
$v_1=(1,-1,1)^T$, $v_2=(-1,-1,-1)^T$, $v_3=(-1,1,1)^T$, $v_4=(1,1,-1)^T$, and the matrix corresponding to $\sigma\in S_4$ maps $v_j$ to $v_{\sigma(j)}$; for instance, $M_{((12),-1)}$ interchanges $v_1$ and $v_2$ and fixes $v_3$ and $v_4$, while $M_{((1234),-1)}$ maps $v_1\mapsto v_2\mapsto v_3\mapsto v_4\mapsto v_1$.
Using the $G$-action in \eqref{eq: actions},  a direct computation of the fixed-point subspace gives
\[
A=t
\begin{pmatrix}
1&-1&1\\
-1&-1&-1\\
-1&1&1\\
1&1&-1
\end{pmatrix},
\qquad t\in\mathbb R.
\]
Taking $t=1,$ we obtain
\[ 
\Phi_{S_4^-\times_{S_4}S_4}( x)
=r^{-3/2}J_{3/2}(s_{11}r)
\begin{pmatrix}
x_1-x_2+x_3\\
-x_1-x_2-x_3\\
-x_1+x_2+x_3\\
x_1+x_2-x_3
\end{pmatrix}.
\]
The rows of $A$ (with $t=1$) are exactly $v_1^T,\dots,v_4^T$, so the condition $\gamma a_j=a_{\sigma(j)}$ holds for both generators, and the $j$-th component of $\Phi_{S_4^-\times_{S_4}S_4}$ is a dipole oriented towards the vertex $v_j$ of the tetrahedron. Moreover, the isotropy group of $\Phi_{S_4^-\times_{S_4}S_4}$ is exactly $H$: if $(\gamma,\sigma)$ fixes it, then $\gamma a_j=a_{\sigma(j)}$ for all $j$, and since $a_1,a_2,a_3$ are linearly independent, $\gamma$ is determined by $\sigma$; hence the isotropy group has at most $24=|H|$ elements.
\vskip0.3cm

The fixed-point subspaces for the other two symmetry types are obtained by computations analogous to the tetrahedral case.
\vskip0.3cm

\noindent \textbf{Symmetry of type $(D_3\times _{D_3}D_3)$:} Notice that $H=D_3\times_{D_3}D_3$ is generated by
\[
\left\langle
\bigl(((12),1),(12)\bigr),
\bigl(((123),1),(123)\bigr)
\right\rangle
\]
with
\[
M_{((12),1)}
=
\begin{pmatrix}
0&0&1\\
0&-1&0\\
1&0&0
\end{pmatrix},
\;\;
M_{((123),1)}
=
\begin{pmatrix}
0&1&0\\
0&0&-1\\
-1&0&0
\end{pmatrix}.
\]
Solving the corresponding fixed-point conditions for $A$, we obtain that the fixed-point space is spanned by
\[
\Phi_{D_3\times_{D_3}D_3}(x)
=
r^{-3/2}J_{3/2}(s_{11}r)
\begin{pmatrix}
-x_2-x_3\\
-x_1+x_2\\
x_1+x_3\\
0
\end{pmatrix}.
\]
\vskip0.3cm
\noindent \textbf{Symmetry of type $(O(2)^-\times D_3)$:}
The generators of $H=O(2)^-\times D_3$ can be written as
\[
(R_z(\tau),1),\;\;
(\mathbf{r},1),\;\;
(1,(123)),\;\;
(1,(12)),
\;\; \tau\in[0,2\pi),
\]
where
\[
R_z(\tau)
=
\begin{pmatrix}
\cos\tau&-\sin\tau&0\\
\sin\tau&\cos\tau&0\\
0&0&1
\end{pmatrix},
\qquad
\mathbf r=\operatorname{diag}(1,-1,1).
\]
Then the corresponding fixed-point space is spanned by
\[
\Phi_{O(2)^-\times D_3}(x)
=
r^{-3/2}J_{3/2}(s_{11}r)
\begin{pmatrix}
-x_3\\
-x_3\\
-x_3\\
3x_3
\end{pmatrix}.
\]
Figure~\ref{fig:symmetries} in the Introduction, computed in Mathematica, displays high-density regions of the components of $u_H$ for the three modes. The Mathematica notebook \texttt{BEC\_eigenmodes\_verified.nb} and animations of the modes of Figure~\ref{fig:symmetries}(a) and~(b) are available at \url{https://github.com/violal1016/Bose-Einstein-Condensate}.
%so $\dim(V_1\otimes W_4)^H=1$. All components are axisymmetric about the $z$-axis. The first three perturbations coincide, and the fourth has the opposite sign and three times their amplitude. For each of these subgroups, the generator $\Phi_H$ is determined only up to a nonzero scalar, which can be absorbed into $\varepsilon$. By~(12) and~(13), these generators belong to the kernel of the Hessian at $(u_{\mu_{11}},\mu_{11})$. The first-order fields $u_H^{(1)}$ satisfy the Neumann boundary condition, and substitution into~(3) at $\mu=\mu_{11}$ leaves a residual of order $\varepsilon^2$. Higher-order corrections are needed to recover the nonlinear solution branches.
\vskip0.6cm
\section{Conclusions}\label{sec:future}

We have shown that the symmetric four-component Gross--Pitaevskii model
necessarily develops spatially inhomogeneous amplitude patterns as the chemical
potential is varied. In the immiscible regime with repulsive inter-component
interactions, and in the symmetric parameter regime of equal chemical
potentials and equal self-interactions, the spatially homogeneous branch
$u_{\mu}$ loses stability at a discrete set of thresholds determined by the
Laplace spectrum. As the chemical potential $\mu$ crosses any critical value
$\mu_{km}$ with $(k,m)\neq(0,0)$, nontrivial, spatially inhomogeneous amplitude
patterns bifurcate from the homogeneous state; these are global bifurcations of
$G$-orbits of solutions, corresponding to density modulations rather than phase
singularities.

%\textcolor{magenta}{We emphasise that we study \emph{amplitude} (density) secondary bifurcations,
%not vortices. The restriction to real-valued solutions is a deliberate choice guided
%by the bifurcation analysis, which shows that only real solutions exist as
%secondary bifurcations through the trivial branch $u_{\mu}$. Vortex nucleation
%can be obtained with similar equivariant degree methods applied to bifurcations
%of vortices and dipole-like solutions from the zero solution
%$u=0$ at different values of $\mu$, as has been done in Contreras
%and Garc\'{\i}a-Azpeitia~\cite{CoGa15}, and Garc\'{\i}a-Azpeitia and
%Pelinovsky~\cite{GaPe17} with an additional rotation of frequency $\Omega$.}

The global branches for $k=0$ have symmetries that imply they are radial
solutions: in the branch with symmetries $(O(3)\times D_{3})$, three
of the four components are equal, while in the branch with symmetries
$(O(3)\times D_{2})$ there are two distinct pairs of equal components. Locally,
these solutions are genuine modulations in the radial variable.

For $k=1$ the solutions in bifurcating branches have their components related
by different groups of rotations and are modulated in both radius and angle, with
a guarantee of at least six families with different symmetries. The most notable
solutions have the symmetry group $(S_{4}^{-}\times_{S_{4}}S_{4})$, where the four
components are related by a tetrahedral symmetry of rotations. The emergence of polyhedral symmetry is a key feature of the four-component setting: for three components only
dihedral symmetry groups appear. Figure~\ref{fig:symmetries} illustrates three of these six symmetry types for $k=m=1$.

Because the equivariant gradient degree is a topological invariant, the bifurcating branches are
robust under small continuous perturbations of parameters: weak symmetry
breaking can deform or split branches but cannot annihilate those guaranteed
by a nontrivial degree change. Consequently, the bifurcating amplitude
patterns persist under small perturbations of the parameters, although
larger perturbations may alter or remove some solutions.

Finally, the same equivariant-degree framework employed here can be adapted
to other domains, boundary conditions, and component numbers, and can be extended to study bifurcations of localised states, metastable phases, or dipole-type solutions. It would also be of interest to complement the first-order perturbations illustrated in Section \ref{sec:visualization} with numerical continuation of the nonlinear bifurcating branches, and to derive explicit reduced equations for secondary bifurcating solutions at other values of $k$, such as $k=2,3$. These directions are left for future work.

\appendix

\section{Spectrum of Laplace operator}

\label{app:laplace}

Consider the Laplace eigenvalue problem on the open unit ball $\Omega\subset\mathbb{R}^{3}$ given by
\begin{equation}%
\begin{cases}
-\triangle u(x)=\lambda^{2}u,\quad x\in\Omega\\
\left.  \frac{\partial u}{\partial n}\right\vert _{\partial\Omega}=0.
\end{cases}
\label{eq:Lap-spec1}%
\end{equation}
In the case $\lambda=0$ we know that the only solutions are the constant
functions. For $\lambda
^{2}>0,$ the spectrum of $-\triangle$ can be studied by employing separation of variables in spherical coordinates $u(r,\theta,\vp)$ in $\mathbb{R}^{3}$, with $r\ge 0$,  $\theta\in [0,2\pi]$ and $\vp\in [0,\pi]$. The derivation of the spectrum largely follows Appendix B of \cite{LCW}, with a key difference: we restrict to the case
$\mathbb R^d$ with $d=3,$ and the system satisfies Neumann rather than Dirichlet boundary conditions.\vs
\noi As a result, by defining $\mathcal K:=\{(0,0)\}\cup \{(k,m)\in \mathbb N\times \mathbb N^+\}$, we obtain the spectrum of $-\triangle,$ which is given by
\begin{equation}
\sigma(-\triangle)=\left\{  s^2_{km}:(k,m)\in\mathcal {K}\right\}, \label{eq:spL}%
\end{equation}
where $s_{00}=0$ and in other cases, $s_{km}$ denotes the $m$-th positive zero of \[
\lambda J_{k+1/2}^{\prime}(\lambda)-\frac{1}{2}J_{k+1/2}(\lambda).
\]
Notice, here $J_{k+1/2}$ stands for the $k+1/2$-th Bessel function of the first kind. Moreover,
we have the following subspace $\mathcal{E}_{km}$ with eigenvalue $s^2_{km}$:
$\mathcal{E}_{00}=\left\langle 1\right\rangle $ and for $(k,m)\in\mathcal K\backslash \{(0,0)\},$ one has
\begin{equation}
\mathcal{E}_{km}=\left\langle r^{-\frac{1}{2}}J_{k+1/2}(s_{km}r)P_{k}^{n}
(\cos\varphi)\Big(a\cos(n\theta)+b\sin(n\theta)\Big):a,\,b\in\mathbb{R},0\leq
n\leq k\right\rangle, \qquad \label{eigenspace}%
\end{equation}
where  $P_k^n$ is the Legendre
function, given by the well-known formula
\begin{equation}\label{eq: Legendre}
P_k^n(s)=\frac{(1-s^2)^{\frac n2}}{2^kk!}\frac{d^{k+n}}{ds^{k+n}}(s^2-1)^k.
\end{equation}

Notice that the eigenspace corresponding to the eigenvalue $s^2_{km}$ is the
subspace generated by $\left\{  \mathcal{E}_{jn}:s_{jn}=s_{km}\right\}  $, and
in the case that $s^2_{km}$ is isotypic simple, the corresponding eigenspace is
$\mathcal{E}_{km}$.
\section{Euler and Burnside Rings}\label{app:Eulerring}
\subsection{Euler and Burnside Rings\label{app:Euler}}

In this section we assume that $G$ stands for a compact Lie group and we
denote by $\Phi(G)$ the set of all conjugacy classes $(H)$ of closed subgroups
$H$ of $G$. For any $(H)\in\Phi(G)$ we denote by $N(H)$ the normalizer of $H$
and by $W(H):=N(H)/H$ the Weyl group of $H$. \vskip.3cm

Notice that $\Phi(G)$ admits a natural order relation given by
\begin{equation}
\label{eq:order}(K)\le(H) \;\; \Leftrightarrow\;\; \exists_{g\in G}\; \;
gKg^{-1}\le H, \;\; \text{ for }\;\; (K),\,(H)\in\Phi(G).
\end{equation}
Moreover, we define for $n=0,1,2,\dots$ the following subsets $\Phi_{n}(G)$ of
$\Phi(G)$
\[
\Phi_{n}(G):=\{ (H)\in\Phi(G): \dim W(H)=n\}.
\]
\vskip.3cm Let us point out that for $(H)$, $(K)\in\Phi(G)$, if $(K)\le(H)$
then $\dim W(K)\ge\dim W(H)$ (see \cite{AED}). \vskip.3cm We denote by
$U(G)={\mathbb{Z}}[\Phi(G)]$ the free ${\mathbb{Z}}$-module generated by
symbols $\Phi(G)$, i.e. an element $a\in U(G)$ is represented as
\begin{equation}
\label{eq:sum-a}a=\sum_{(H)\in\Phi(G)} n_{H} \, (H), \quad n_{H}\in
{\mathbb{Z}},
\end{equation}
where the integers $n_{H}=0$ except for a finite number of elements
$(H)\in\Phi(G)$. For such element $a\in U(G)$ and $(H)\in\Phi(G)$, we will
also use the notation
\begin{equation}
\label{eq:coeff-H}\text{coeff}^{H}(a)=n_{H},
\end{equation}
i.e. $n_{H}$ is the coefficient in \eqref{eq:sum-a} standing by $(H)$. We also
put
\begin{equation}
\label{eq:Phi(a)}\Phi(a):=\{(H)\in\Phi(G): \text{coeff}^{H}(a)\not =0\}.
\end{equation}
Clearly, the set $\Phi(a)$ is finite and we denote by $\text{max}(a)$ the set
of all maximal (with respect to the order \eqref{eq:order}) elements in
$\Phi(a)$, i.e.
\begin{equation}
\label{eq:max-in-a}\max(a):=\{(H)\in\Phi(a): \;(H) \text{ is maximal in }%
\Phi(a)\}.
\end{equation}
We will sometimes write ${\max}^G(a)$ in order to indicate that the order relation is taken from $\Phi(G)$.
\vskip.3cm

\begin{definition}
\textrm{\label{def:EulerRing} (cf. \cite{tD}) The
\textit{multiplication} on $U(G)$ is defined as follows: Let $(H)$,
$(K)\in\Phi(G)$ be generators. Then,
\begin{equation}
(H)\ast(K):=\sum_{(L)\in\Phi(G)}n_{L}(L),\label{eq:Euler-mult}%
\end{equation}
where
\begin{equation}
n_{L}:=\chi_{c}((G/H\times G/K)_{L}/N(L)).\label{eq:Euler-coeff}%
\end{equation}
Note,  $(G/H\times G/K)_{L}$ denotes the set of elements in $G/H\times G/K$ that are
fixed exactly by $L$. $\chi_c(\cdot)$ denotes the Euler Characteristic (For its precise definition, see Section 3 of \cite{DK} and \cite{Spa}).} Moreover, the multiplication is extended linearly to the entire Euler ring
$U(G)$. Then the free ${\mathbb{Z}}$-module $U(G)$ associated with
multiplication (\ref{eq:Euler-mult}) is called the \textit{Euler ring} of $G$.
\end{definition}

\vskip.3cm It is easy to notice that $(G)$ is the unit element in $U(G)$, i.e.
$(G)*a=a$ for all $a\in U(G)$. \vskip.3cm

\begin{lemma}
\label{lem:inv-1} Assume that $a\in U(G)$ is an invertible element and
$(H)\in\Phi(G)$. Then
\[
\text{\text{\rm coeff}}^{H}((H)*a)\not =0.
\]

\end{lemma}

\begin{proof}
Suppose that
\[
a=\sum_{(L)\in\Phi(a)} n_{L} \, (L).
\]
Then
\[
(H)*a=\sum_{(K)\in\Phi((H)*a)} m_{K}\, (K), \text{ and formula
(\ref{eq:Euler-mult}) implies that} \;\; (H)\ge(K).
\]
Assume for contradiction that $(H)>(K)$ for all $(K)\in\Phi((H)*a)$. Then, by
exactly the same argument we have
\[
(H)*a*a^{-1}=\sum_{(L)\in\Phi((H)*a*a^{-1})} n_{L} \, (L), \quad\text{ where }
(H)>(L),
\]
which is a contradiction with the fact that
\[
(H)*a*a^{-1}=(H)*(G)=(H).
\]

\end{proof}

\vskip.3cm Lemma \ref{lem:inv-1} implies the following: \vskip.3cm

\begin{corollary}
\label{cor:inv} Let $a$, $b\in U(G)$ be such that $a$ is an invertible element
in $U(G)$ and $(H)\in\text{\textrm{max}}(b)$. Then $\text{\textrm{coeff}}%
^{H}(a*b)\not =0$.
\end{corollary}

\vskip.3cm Now, let $\Phi_{0}(G)=\{(K)\in\Phi(G):$ \textrm{dim\thinspace
}$W(K)=0\},$ and $A(G)={\mathbb{Z}}[\Phi_{0}(G)]$ denote a free
${\mathbb{Z}}$-module with basis $\Phi_{0}(G).$ The multiplication on
$A(G)$ is defined as the restriction of the multiplication on $U(G)$ to $A(G)$, i.e. for $(K)$,
$(H)\in\Phi_{0}(G)$ let
\begin{equation}
(K)\cdot(H)=\sum_{(L)}n_{L}(L),\qquad(K),(H),(L)\in\Phi_{0}(G),\text{ where}
\label{eq:multBurnside}%
\end{equation}
\begin{equation}
n_{L}=\chi((G/K\times G/H)_{L}/N(L))=|(G/K\times G/H)_{L}/N(L)|.
\label{eq:CoeffBurnside}%
\end{equation}
Notice $\chi(\cdot)$ denotes the usual Euler characteristic (see \cite{DK}). Now, equipped with the multiplication defined in \eqref{eq:CoeffBurnside}, $A(G)$ becomes a ring, known as the
\textit{Burnside ring} of $G$. As shown in \cite{AED}, the coefficients in \eqref{eq:CoeffBurnside} can be computed using the recursive formula:
\begin{equation}
n_{L}=\frac{n(L,H)|W(H)|n(L,K)|W(K)|-\sum_{(L^{'})>(L)}n(L,L^{'})n_{L^{'}}|W(L^{'})|}{|W(L)|}, \label{eq:rec-coef}%
\end{equation}
where $(H),$ $(K),$ $(L)$ and $(L^{'})$ are taken from $\Phi_{0}
(G)$ and \[
n(L,H)=\Big|\frac{N(L,H)}{N(H)}\Big|,
\]
where
\begin{align*}
&
N_{G}(L,H)=\Big\{
g\in G:gLg^{-1} \subset H\Big\} , \;N_{G}(L,H)/H=\Big\{
Hg: g\in N_{G}(L,H)\Big\}.
\end{align*}
 Notice that $A(G)$ is a ${\mathbb{Z}}$-submodule
of $U(G)$, but it is \textbf{not} a subring of $U(G)$ in general.

\vskip.3cm Now for $(H)\in
\Phi(G),$ we define $\pi_{0}:U(G)\rightarrow A(G),$
\begin{equation}
\pi_{0}((H))=
\begin{cases}
(H) & \text{ if }\;(H)\in\Phi_{0}(G),\\
0 & \text{ otherwise.}%
\end{cases}
\label{eq:pi_0-homomorphism}%
\end{equation}
Then $\pi_{0}$ here is a ring
homomorphism {(cf. \cite{BKR})}, which also implies that
\[
\pi_{0}((K)\ast(H))=\pi_{0}((K))\cdot\pi_{0}((H)),\qquad(K),(H)\in\Phi(G).
\]
A distinction between the generators $(H)$ in $U(G)$ and those of $A(G)$ is demonstrated by the following result (cf. \cite{tD}, Proposition 1.14, page 231)
\begin{proposition}
\label{pro:nilp-elements} Let $(H)\in\Phi_{n}(G)$.

\begin{itemize}
\item[(i)] If $n>0$, then $(H)^{k}=0$ in $U(G)$ for some $k\in\mathbb{N}$,
i.e. $(H)$ is a nilpotent element in $U(G)$;

\item[(ii)] If $n=0$, then $(H)^{k}\not =0$ for all $k\in{\mathbb{N}}$.
\end{itemize}
\end{proposition}
\begin{corollary}
    If $\alpha=n_1(L_1)+n_2(L_2)+\cdots +n_k(L_k),$ where $\dim W(L_j)> 0,$ then there exists $n\in \mathbb N$ s.t. $\alpha^n=0.$
\end{corollary}
\begin{proof}
    By induction w.r.t. $k\in \mathbb N$, clearly for $k=1,$ it is exactly the statement of Proposition \ref{pro:nilp-elements}. Suppose that the statement is true for $k\geq 1,$ and will show that it is also true for $k+1.$ Indeed, we have
    \[
\alpha^1=n_1(L_1)+n_2(L_2)+\cdots +n_k(L_k)+n_{k+1}(L_{k+1})
    =\alpha+n_{k+1}(L_{k+1}),
    \]
so \[
\alpha^m=\sum_{l=0}^{m}C_m^l \alpha^l n_{k+1}^{m-l}(L_{k+1})^{m-l}.
\]
Let $k$ be given by Proposition \ref{pro:nilp-elements}, for $L:=L_{k+1},$ then for $m\ge n+k,$ one has
\[
\alpha^l (L)^{m-l}=0
\]
    \end{proof}
\vskip.3cm \subsection{Euler Ring Homomorphism} Let $\psi:{G^{\prime}
}\rightarrow G$ be a homomorphism between two compact Lie groups. Then, it induces a left action of ${G^{\prime}} $ on $G,$ defined by
\[
{g^{\prime}}x:=\psi({g^{\prime}})x,\]
and similarly, a right action by
$x{\ g^{\prime}}:=x\psi({g^{\prime}}).$ Particularly, for any subgroup $H\leq G$, the map $\psi$ induces a left $G^{\prime}$-action
on the homogeneous space $G/H,$ where the stabilizer of a point $gH\in G/H$ is given by
\begin{equation}
G_{gH}^{\prime}=\psi^{-1}(gHg^{-1}). \label{eq:grad-iso-rel}%
\end{equation}
Consequently, $\psi$ induces a ring homomorphism $\Psi:U(G)\rightarrow U({G^{\prime}}),$ given by
\begin{equation}
\Psi((H))=\sum_{({H^{\prime}})\in\Phi({G^{\prime}},G/H)}n_{H^{'}}({H^{\prime}}), \quad \text{where }
n_{H^{\prime}}=\chi_{c}
\Big((G/H)_{({H^{\prime}})}/{G^{\prime}}\Big).
\label{eq:RingHomomorphism}%
\end{equation}
\vs
\begin{proposition}
\label{prop:RingHomomorphism} \textrm{(cf. \cite{BKR,tD})} The map $\Psi$
defined by (\ref{eq:RingHomomorphism}) is a homomorphism of Euler rings.
\end{proposition}
\vs

Notice that if $G^{'}\le G$ and $H\le G$,  then $(H')\in \Phi(G';G/H)$ if and only if
\[
H'=g H g^{-1} \cap G' \quad \text { for some } \; g\in G.
\]
Indeed,
\[
gH \in (G/H)_{(H^{\prime})} \Leftrightarrow (G^{\prime}_{gH})=(H^{'}),
\]
and
\begin{align*}
G^{'}_{gH}:
& =\left\{
g^{\prime}\in G^{\prime}: g^{\prime}gH=gH
\right\} =\left\{ g^{\prime} \in G^{\prime}: g^{-1}g^{\prime} g \in H\right\}\\
&= \left \{ g^{\prime} \in G^{\prime}: g^{\prime}\in gHg^{-1}\right\}=gHg^{-1}\cap G^{\prime}
\end{align*}
Therefore, $(G/H)_{H^{\prime}}:=\Big\{gH:g\in G \text{ and } gHg^{-1}\cap G^{\prime}=H^{\prime}\Big\}.$ \vs

\section{Brouwer $G$-Equivariant Degree and $G$-Equivariant Gradient Degree}

\label{app:gradient}
In what follows, we assume that $G$ is a compact Lie group.
\subsection{Brouwer $G$-Equivariant Degree and Its Computation}
We briefly review some essential definitions and properties of Brouwer $G$-equivariant degree in this section. For more details, we refer to Chapter 3 of \cite{newBook}.\vs Let $V$ be an orthogonal representation of a compact Lie group $G$, and let $\Omega\subset V$ be an
open, bounded, $G$-invariant subset. A $G$-equivariant map $f:V\to V$ is said to be \textit{$\Omega$-admissible} if $\forall x\in
\partial\Omega,$ $f(x)\neq0.$ Therefore, the pair $(f,\Omega)$ is called
\textit{$G$-admissible pair}. Let $\mathcal{M}^{G}(V,V)$ be the collection of all admissible $G$-pairs, define $\mathcal{M}^{G}:=\bigcup_{V}\mathcal{M}
^{G}(V,V)$, where the union is for all finite-dimensional orthogonal $G$-representations
$V$.
\begin{theorem}
\label{def:GpropDeg} Let $G$-$\deg:\mathcal{M}^{G}\to
A(G)$ be a map assigning to each admissible
$G$-pair
$(f,\Omega)$ an element of Burnside Ring
$A(G)$, given by
\begin{equation}
\label{eq:G-deg0}G\mbox{\rm -}\deg(f,\Omega)=\sum_{(H_{i})\in\Phi_{0}
(G)}n_{H_{i}}(H_{i})= n_{H_{1}}(H_{1})+\dots+n_{H_{m}}(H_{m}).
\end{equation}
It satisfies the properties of additivity, homotopy, normalization, as well as existence, product, suspension, recurrence formula, etc. (see \cite{newBook} for details).
We call $G\mbox{\rm -}\deg(f,\Omega)$ the \textit{$G$-equivariant
degree} (or simply \textit{$G$-degree}) of $f$ on $\Omega$.
\end{theorem}\vs
Now, let us consider the $G$-isotypic decomposition of $V$:
$V = \bigoplus_{i=0}^{r} V_i,$
where each $V_i$ is modeled on the irreducible $G$-representation $\cV_i$ for $i=0, 1, 2,\dots, r.$ Let $B(V)$ be the unit ball in $V$.
\begin{definition} The Brouwer $G$-equivariant degree
\begin{equation} \label{eq:basicdeg}
\deg_{\cV_i} := G\text{-}\deg(-\id, B(\cV_i)),
\end{equation}
is called the {\it $\cV_i$-basic degree} (or simply {\it basic degree}), and it can be computed by: $\deg_{\mathcal{V} _{i}}=\sum_{(K)}n_{K}(K),$
where
\begin{align}\label{eq:formula}
n_{K}=\frac{(-1)^{\dim\mathcal{V} _{i}^{K}}- \sum_{K<L}{n_{L}\, n(K,L)\, \left|  W(L)\right|  }}{\left|  W(K)\right|  }.
\end{align}
\end{definition}

\begin{lemma}\label{le:basic_coefficient}
If for $ (K_o) \in \Phi_0(G)$, one has $\text{coeff}^{L_o} (\deg_{\cV_i})$ is a leading coefficient of $\deg_{\cV_i}$, then $\dim (\cV_i ^{K_o})$ is odd and
\[
\text{coeff}^{K_o}(\deg_{\cV_i})=
\begin{cases}\label{eq:a0}
-1 & \text{if} \; |W(K_o)|=2,\\
-2 & \text{if}\; |W(K_o)|=1;\\
\end{cases}
\]
\end{lemma}
\begin{lemma}\label{le:involutive}
    For each  $\mathcal V_i$,  the corresponding basic degree $\deg_{\mathcal V_i} \in A(G)$ is an involution in the Burnside ring. It satisfies
    \[
    (\deg_{\cV_i})^2=\deg_{\cV_i} \cdot \deg_{\cV_i}=(G).
    \]\vs
\end{lemma}
\vs

\vskip.3cm

%---

\subsection{$G$-Equivariant Gradient Degree and Its Computation}
In this section, we recall the essential information of our main tool, $G$-Equivariant Gradient Degree, with particular focus on computation of basic gradient degree for $G:=O(3)\times \Gamma (\Gamma\;\; \text{is finite}).$ Note, in our setting $\Gamma=S_4.$ For more details, we refer to Section 6 of \cite{DK}.\vs
 Let $V$ be an orthogonal $G$-representation, and denote by $C^{2}_{G}
(V,\mathbb{R})$ the space of $G$-invariant real $C^{2}$-functions on $V$. Given
$\varphi\in C^{2}_{G}(V,\mathbb{R})$ and open, bounded, $G$-invariant subset $\Omega\subset V$ such that $\nabla\varphi(x)\not =0$ for $x\in\partial\Omega$, the pair $(\nabla\varphi,\Omega)$ is said to be \textit{$G$-gradient
$\Omega$-admissible}. Let $\mathcal{M}^{G}_{\nabla}(V,V)$ be the set of all
such $G $-gradient $\Omega$-admissible pairs, define
\[
\mathcal{M}^{G}_{\nabla}:=\bigcup_{V}\mathcal{M}^{G}_{\nabla}(V,V).
\]
Then we have the following result for $G$-Equivariant gradient degree.
\begin{theorem}
\label{thm:Ggrad-properties}{(cf. \cite{G,DK})} There is a unique map
$\nabla_{G}\mbox{\rm -}\deg:\mathcal{M}_{\nabla}^{G}\to U(G)$, assigning to each $(\nabla\varphi,\Omega)\in\mathcal{M}^{G}_{\nabla}$ an element
in $U(G),$ called the
\textit{$G$-equivariant gradient degree} of $\nabla\varphi$ on $\Omega.$
\begin{equation}
\label{eq:grad-deg}\nabla_{G}\mbox{\rm -}\deg(\nabla\varphi,\Omega)=
\sum_{(H_{i})\in\Phi(\Gamma)}n_{H_{i}}(H_{i})= n_{H_{1}}(H_{1})+\dots
+n_{H_{m}}(H_{m}).
\end{equation}
It satisfies the following properties:

\begin{enumerate}
%\item \renewcommand\labelenumi{\rm\bf($\nabla$\arabic{enumi})}
\item \textrm{\textbf{(Existence)}} If $G\mbox{\rm -}\deg(f,\Omega)\neq0$,
i.e., \eqref{eq:G-deg0} contains a non-zero coefficient $n_{H_{i}}$, then
$\exists_{x\in\Omega}$ such that $f(x)=0$ and $(G_{x})\geq(H_{i})$.
\item \textrm{\textbf{(Additivity)}} Let $\Omega_{1}$ and $\Omega_{2}$ be two
disjoint open $G$-invariant subsets of $\Omega$ such that $(\nabla
\varphi)^{-1}(0)\cap\Omega\subset\Omega_{1}\cup\Omega_{2}$. Then,
\[
\nabla_{G}\mbox{\rm -}\deg(\nabla\varphi,\Omega)= \nabla_{G}\mbox{\rm -}\deg
(\nabla\varphi,\Omega_{1})+ \nabla_{G}\mbox{\rm -}\deg(\nabla\varphi
,\Omega_{2}).
\]

\item \textrm{\textbf{(Homotopy)}} If $\nabla_{v}\Psi:[0,1]\times V\to V$ is a
$G$-gradient $\Omega$-admissible homotopy, then
\[
\nabla_{G}\mbox{\rm -}\deg(\nabla_{v}\Psi(t,\cdot),\Omega)=\mathrm{constant}.
\]

\item \textrm{\textbf{(Normalization)}} Let $\varphi\in C^{2}_{G}
(V,\mathbb{R})$ be a special $\Omega$-Morse function such that $(\nabla
\varphi)^{-1}(0)\cap\Omega=G(v_{0})$ and $G_{v_{0}}=H$. Then,
\[
\nabla_{G}\mbox{\rm -}\deg(\nabla\varphi,\Omega)= (-1)^{{m}^{-}(\nabla
^{2}\varphi(v_{0}))}\cdot(H),
\]
where ``$m^{-}(\cdot)$'' stands for the total dimension of eigenspaces for
negative eigenvalues of a (symmetric) matrix.

\item \textrm{\textbf{(Functoriality Property)}} Let $V$ be an orthogonal
$G$-representation, $f:V\rightarrow V$ a $G$-gradient $\Omega$-admissible map,
and $\psi:G_{0}\hookrightarrow G$ an embedding of Lie groups. Then, $\psi$
induces a $G_{0}$-action on $V$ such that $f$ is an $\Omega$-admissible
$G_{0}$-gradient map, and the following equality holds
\begin{equation}
\Psi\lbrack\nabla_{G}\mbox{\rm -}\deg(f,\Omega)]=\nabla_{G_{0}}
\mbox{\rm -}\deg(f,\Omega), \label{eq:funct-G}%
\end{equation}
where $\Psi:U(G)\rightarrow U(G_{0})$ is the homomorphism of Euler rings
induced by $\psi$.

\item \textrm{\textbf{(Reduction Property)}} Let $V$ be an orthogonal
$G$-representation, $f:V\to V$ a $G$-gradient $\Omega$-admissible map, then
\begin{equation}
\label{eq:red-G}\pi_{0}\left[  \nabla_{G}\mbox{\rm -}\deg(f,\Omega)\right]
=G\mbox{\rm -}\deg(f,\Omega).
\end{equation}
where the ring homomorphism $\pi_{0}:U(G)\to A(G)$ is given by \eqref{eq:pi_0-homomorphism}.
\end{enumerate}
\end{theorem}
For details of some other properties, one can see \cite{DK}. We next introduce its computation via standard linearization technique.\vs Let $A:V\to V$ be a $G$-equivariant symmetric isomorphism of $V$ and $A:=\nabla\varphi$, where $\varphi(x)=\frac12 Ax\bullet x$. Consider the
$G$-isotypical decomposition of $V:$ $V=\bigoplus_{i}V_{i},\quad V_{i}\;\;\mbox{modeled on}\;\mathcal{V}_{i}.$ We denote by $\sigma(A)$ the
spectrum of $A$ and $\sigma_{-}(A)$ its negative spectrum. Let  $E_{\mu}(A)$ be the eigenspace corresponding to $\mu
\in\sigma(A)$. Put $\Omega:=\{x\in V: \|x\|<1\}$. Then, $A$ is $\Omega
$-admissibly homotopic (in the class of gradient maps) to a linear operator
$A_{o}:V\to V$ such that $A_{o}|_{E_{\mu}(A)}=-\text{\textrm{Id}}$ for $\mu\in
\sigma_{-}(A)$ and $A_{o}|_{E_{\mu}(A)}=\text{\textrm{Id}}, \;\mu\in\sigma(A)\setminus\sigma_{-}(A).$ \vs \noi Suppose $\mu\in\sigma_{-}(A)$ and
denote by $m_{i}(\mu)$ the integer
$m_{i}(\mu):=\dim(E_{\mu}(A)\cap V_{i})/\dim\mathcal{V}_{i},$
which is called the \textit{$\mathcal{V}_{i}$-multiplicity} of $\mu$. Since
$\nabla_{G}\mbox{\rm -}\deg(\text{\textrm{Id}},\mathcal{V}_{i})=(G)$ is the
unit element in $U(G)$, we derive that
\begin{equation}
\label{eq:grad-lin}\nabla_{G}\mbox{\rm -}\deg(A,\Omega)= \prod_{\mu\in
\sigma_{-}(A)}\prod_{i} \left[  \nabla_{G}\mbox{\rm -}\deg(-\text{\textrm{Id}
},B(\mathcal{V}_{i}))\right]  ^{m_{i}(\mu)}.
\end{equation}
\begin{definition} \label{de:basicequi}
\textrm{\ Assume that $\mathcal{V}_{i}$ is an irreducible $G$-representation.
Then, the $G$-equivariant gradient degree:
\[
\nabla_{G}\text{-}\deg_{\mathcal{V}_{i}}:=\nabla_{G}\mbox{\rm -}\deg(-\text{\text{\rm Id}}
,B(\mathcal{V}_{i}))\in U(G)
\]
is called the \textit{gradient $G$-equivariant basic degree} for
$\mathcal{V}_{i}$. }
\end{definition}
\begin{proposition}
\label{prop:invertible} The basic gradient degrees $\nabla_{G}\text{-}\deg_{\mathcal{V}_{i}}$ are invertible elements in $U(G)$.
\end{proposition}

\begin{proof}
Let $a:=\pi_{0}(\nabla_G\text{\textrm{-deg\,}}_{\cV_{i}})$, then $a^{2}=(G)$ in $A(G)$ (see Lemma \eqref{le:involutive}),
which implies that $(\nabla_{G}\text{\textrm{-deg\,}}_{\mathcal{V}_{i}})^{2}=(G)-\alpha$,
where for every $(H)\in\Phi_{0}(G)$ one has coeff$^{H}(\alpha)=0$. It is
sufficient to show that $(G)-\alpha$ is invertible in $U(G)$. Since (by
Proposition \ref{pro:nilp-elements}) for sufficiently large $n\in{\mathbb{N}}%
$, $\alpha^{n}=0$, one has
\[
\big((G)-\alpha)\sum_{n=0}^{\infty}\alpha^{n}=\sum_{n=0}^{\infty}\alpha
^{n}-\sum_{n=1}^{\infty}\alpha^{n}=(G),
\]
where $\alpha^{0}=(G)$
\end{proof}\vs
In order to easier formulate our arguments, we put
\begin{equation}\label{eq:maxG}
{\max}^G_0(\cV_i):={\max}^G\left(\pi_0\left(\nabla_G\text{-deg}_{\cV_i}-(G)\right)\right)=
{\max}^G\left(G\text{-deg}_{\cV_i}-(G)\right).
\end{equation}\vskip 0.3cm
\noi{\bf Computations of basic gradient degrees}:
We are motivated by the computations of the basic gradient degrees for the group $G:=O(3)\times \Gamma$ (where $\Gamma$ is finite) using the reduction to the subgroup $G^{\prime}=S_4^p \times \Gamma$ by applying the Euler homomorphism $\Psi:U(  O(3)\times \Gamma)\to U(S_4^p \times \Gamma)$. Since $G':=S_4^p \times \Gamma$ is finite, $U(S_4^p \times \Gamma)=A(S_4^p \times \Gamma)$ and
$\nabla_{G'}\text{-deg}(f,\Omega)=G'\text{-deg}(f,\Omega)$ (for a $G'$-representation $U$, $f=\nabla \vp$,  and $\Om\subset U$), thus
all the basic degrees can be effectively computed using GAP software. Notice that by the  functoriality property of the gradient degree
\begin{equation*}
\Psi\left[\nabla_{G}\mbox{\rm -}\deg(-\id, B_1(0))\right]=\nabla_{G'}
\mbox{\rm -}\deg(-\id,B_1(0)), %
\end{equation*}
where $B_1(0)\subset \mathcal V_i$ denotes the unit ball in an irreducible $G$-representation $\mathcal V_i$, one can use the Euler ring homomorphism $\Psi$ to derive information of $\nabla_{G}\mbox{\rm -}\deg(-\id, B_1(0))$.
 If, by chance, $\cV_i=:\cU_j$ is also an irreducible $G'$-representation, then, in particular, we have
\[
\Psi\left[ \nabla_G\text{\textrm{-deg\,}}_{\cV_{i}}  \right] =\nabla_{G'}\text{\textrm{-deg\,}}_{\cU_{j}}
\]
In our case we are only interested in computing the coefficients of
\[\pi_0( \nabla_G\text{\rm-deg}_{\cV_{i}})=\gdeg(-\id,B_1(0))=\deg_{\cV_i}\]
corresponding to the orbit types in ${\max}_0^G(\cV_i)$.
\vs

\begin{remark}\label{rm:23}
   \rm Let $V$ be a $G_1$-representation and $W$ be a $G_2$-representation. Suppose $(K)$ is maximal orbit type in $V$ and $(H)$ is maximal orbit type in $W$ such that $\dim V^K=\dim W^H=1.$ Then both $N(K)$ acts on $V^K$, $N(H)$ acts on $W^H$ non-trivially if $N(K)\neq K, N(H)\neq H$.\vs
So, there exists
\begin{align*}
    \varphi_1&:N(K)\to O(1)=\mathbb Z_2,\;\;\; \ker \varphi_1=K\\
    \varphi_2&:N(H)\to O(1)=\mathbb Z_2,\;\;\; \ker \varphi_2=H
\end{align*}
and for $x\neq 0, x\in V^K$ and $y\neq 0, y\in W^H,$ one has
\begin{align*}
    (k,h)\in K\times H\implies (k,h)(x\otimes y)=kx\otimes hy=\varphi_1(k)x\otimes \varphi_2(h)y
\end{align*}
Therefore,
\[
(k,h)(x\otimes y)=(x\otimes y) \iff \varphi_1(k)=\varphi_2(h).
\]
i.e. for $G=G_1\times G_2,$
\[
G_{x\otimes y}=N(K)^{\varphi_1}\times ^{\varphi_2}N(H)
=N(K)^K\times ^H N(H).
\]
By maximality of $(K)$ and $(H)$, $N(K)^K\times ^H N(H)$ is also maximal.
\end{remark}
The computations can be illustrated on the following example:
\vs

\begin{example}\label{ex:basic_degree}\rm  Consider the group $G:=O(3)$ and $G':=S_4^p$ with ${\bm\cV}_i:=\cV_1$ being the natural representation of $O(3)$ and $\cU_j:=\cW_{3}^-$. Then, by direct computations through G.A.P, one has
\begin{equation}
\nabla_{G'}\text{\textrm{-deg\,}}_{\cW_{3}^-}=\deg_{\cW_3^-}=(S_4^p)-(D_4^z)-(D_3^z)-(D_2^d) +2(D_1^z)+(\bz_2^-)-(\bz_1),
\end{equation}
where $(D_4^z)$, $(D_3^z)$, $(D_2^d)\in {\max}_0^{S_4^p}(\cW_3^-)$. \vs
\noi Moreover, one also has
\begin{equation}
\nabla_G\text{\textrm{-deg\,}}_{\cV_{1}}=\deg_{{\cV}_1}=(O(3))-(O(2)^-),
\end{equation}
where $(O(2)^-)\in {\max}_0^{O(3)}(\cV_1)$. Indeed, the $SO(3)$-basic degrees are well-defined in \cite{AED}. However, in order to understand better the structure of $O(3)$, recall that
\[
\Phi_0(SO(3))=\Big\{(SO(3)),(A_5),(S_4), (A_4), (D_n), (SO(2)), (O(2))\Big\}
\]
Notice that $O(2)\leq SO(3)$ consists of matrices:
\begin{align*}
    A&=\begin{bmatrix}
        e^{i\theta}&0\\
        0&1
    \end{bmatrix}\quad \quad
     B=\begin{bmatrix}
        e^{i\theta}\kappa&0\\
        0&-1
    \end{bmatrix},\;\;\; \;\;
\end{align*}
where
\begin{equation*}
e^{i\theta}=\begin{bmatrix}
        \cos\theta & -\sin \theta\\
        \sin \theta & \cos\theta
    \end{bmatrix},\;\;\;\;
\kappa=\begin{bmatrix}
    1&0\\0&-1\\
\end{bmatrix}.
\end{equation*}
Take $p=(0,0,1)\in S^{2}$, and $\tilde{G}=SO(3).$ Then one can easily observe that
\[
\tilde{G}_p=SO(2)
\]
is a rotation about $z$-axis and $Bp=-p.$ \vs
For $O(3)=SO(3)\times \mathbb Z_2,$ one has that $\mathbb Z_2=\{\id,-\id\}$ and every subgroup $H\leq O(3)$ can be described as a product group $ H=\mathcal H\times \mathbb Z_2$ or a twisted subgroup of the product $SO(3)\times \mathbb Z_2,$ i.e.
\[
 H\leq SO(3)\times \mathbb Z_2,\;\; H=\mathcal H^{\varphi},
\]
where $\mathcal H\leq SO(3)$ and $\varphi: \mathcal H\to \mathbb Z_2$ homomorphism.\vs
%To characterize twisted subgroups $H$, one can refer to Appendix \eqref{sec:notation}.\vs
\noi Now take $p=(0,0,1)\in S^{2}$ and $G=O(3),$ then one has
$G_p=O(2)^-.$
We know that $\tilde{G}_p\subset G_p,$ but it is clear that for $B\in O(2)\setminus SO(2),$ one has $\varphi(B)=-1,$ thus $B\in G_p$, and
\[
SO(2)\times \{\id\} \cup \Big(O(2)\setminus SO(2)\Big)\times \{-\id\}=O(2)^-.
\]
Now one needs to derive the leading coefficient $\text{coeff}^{O(2)^-}(\deg_{\cV_1}).$ Notice that $\dim \cV_1^{O(2)^-}=1,$ applying formula \eqref{eq:formula}, one has
\[
\nabla _G\text{\textrm{-deg\,}}_{\cV_{1}^-}=\deg_{{\cV}_1}=(O(3))+x(O(2)^-),
\]
where $x=\frac{(-1)^{\dim \cV_1^{O(2)^-}}-1}{|W(O(2)^-)|}=\frac{-2}{2}=-1.$ \vs
Then, clearly, for $\psi:S_4^p\to O(3)$, one has
\[
\Psi(O(3))=S_4^p,\quad \Psi((O(2)^-)=(D_4^z)+(D_3^z)+(D_2^d)+\beta.
\]
Indeed, notice that the conjugacy classes $(D_4^z),(D_3^z), (D_2^d)$ are not orbit types in $\cV_1.$ Since $\cV_1$ has only two orbit types $(O(3))$ and $(O(2)^-),$ by homomorphism theorem,
\[
\Psi \Big(G\mbox{\rm -}\deg(-\id,B(\cV_1)\Big)=G'\mbox{\rm -}\deg(-\id, B(\cW_3^-)),
\]
Given that $(O(2)^-)\in \max^G(\mathcal V_i)$ and $(H')\in {\max}^{G'}(\mathcal{U}_{j})$ (see \eqref{eq:RingHomomorphism}), where $H':=gO(2)^{-}g^{-1}\cap G'$. Then by the maximality of $(H')$ (see Lemma \eqref{le:basic_coefficient}), one has
\[
\textrm{coeff}^{H'}(\nabla_{G'}\mbox{\rm -}\deg_{\mathcal{\cU}_{j}})=
    \begin{cases}
    -2 \;\; \text{if  } W_{G'}(H')=\mathbb{Z}_1\\
    -1 \;\; \text{if  } W_{G'}(H')=\mathbb{Z}_2
    \end{cases}
\]

\vs
In the case $G:=O(3)\times S_4$ and $G':=S_4^p\times S_4$, assume that
${\bm \cV}_i:=\cV_1\otimes \cW_4$ and $\cU_j:=\cW_{3}^-\otimes \cW_{4}$.
Then we have
\begin{align*}
{\max}^{S_4}({\cW_4})&=\{(D_3), (D_2)\},\;\;
{\max}^{S_4^p}({\cW_3^-})=\{(D_4^z), (D_3^z), (D_2^d)\},\\
{\max}^{O(3)}(\cV_1)&=\{(O(2)^-)\}.
\end{align*}
Take, $(H)\in {\max}^{S_4}({\cW_4})$ and $(K)\in {\max}^{S_4^p}({\cW_3^-})$. Since
\[
\dim \cW_4^H=1 \quad \text{ and } \quad  \dim (\cW_3^-)^K=1
\]
then $(\cW_3^-)^K\otimes \cW_4^H$ is a fixed point space of the subgroup
\[
\mathscr H':=\begin{cases}
N(K)^K\times ^HN(H),& \text{ if } \; W(K)=W(H)=\bz_2,\\
K\times H, &\text { if } \; \text{ otherwise},
\end{cases}
\]
 and by maximality assumptions, $(\mathscr H')\in {\max}^{S_4^p\times S_4}(\cW_{3}^-\otimes \cW_{4})$ (see Remark \ref{rm:23}). This leads us to the following maximal orbit types in ${\max}^{S_4^p\times S_4}(\cW_{3}^-\otimes \cW_{4})$:
\begin{align*}
({D_4^p}^{D_4^z}\times ^{D_2} D_4), \quad  ({D_3^p}^{D_3^z}\times ^{D_2} D_4), \quad
 ({D_2^p}^{D_2^d}\times ^{D_2} D_4)\\
 (D_4^z\times D_3), \quad  (D_2^d\times D_3), \quad (D_3^z\times D_3).
\end{align*}
The orbit type $(D_1)$ in $\mathcal W_4$ is not maximal, however the orbit types
\[
({D_4^p}^{D_4^z}\times ^{D_1} D_2),\quad ({D_3^p}^{D_3^z}\times ^{D_1} D_2), \quad ({D_2^p}^{D_2^d}\times ^{D_1} D_2)
\]
are maximal in $\cW_{3}^-\otimes \cW_{4}$, i.e.
\[
\{({D_4^p}^{D_4^z}\times ^{D_1} D_2), ({D_3^p}^{D_3^z}\times ^{D_1} D_2), ({D_2^p}^{D_2^d}\times ^{D_1} D_2)\}\subset
{\max}^{S_4^p\times S_4}(\cW_{3}^-\otimes \cW_{4}).
\]
In addition, in a similar way, one can notice that $(\bz_2^-)$ is an orbit type in $\cW_3^-$ which is not maximal but there are two (non-conjugate) homomorphisms $\vp_1$, $\vp_2:D_4^p\to D_4$ with $\ker \vp_1=\ker \vp_2=\bz_2^-$,  such that
\[
(({D_{4}^{p}}^{^{{\mathbb{Z}}_{2}^{-}}}\times{_{D_{4}}}D_{4})_1:=({D_{4}^{p}}^{^{\vp_1}}\times _{D_4}D_4), \quad ({D_4^p}^{^{\bz_2^-}}\times _{D_4}D_4)_2:=({D_4^p}^{^{\vp_2}}\times _{D_4}D_4)
\]
are also maximal in $\cW_{3}^-\otimes \cW_{4}.$
\vs

Finally, by inspection, one can easily recognize that
\[
(S_4^-\times _{S_4}S_4), \quad \text{ and } \quad (D_3\times _{D_3}D_3)
\]
are also maximal in $\cW_{3}^-\otimes \cW_{4}$.
\vs
%Now, consider the subgroups $\mathscr H'$:
%\[
%{D_4^p}^{D_4^z}\times ^{D_1} D_2,\; {D_3^p}^{D_3^z}\times ^{D_1} D_2, \; {D_2^p}^{D_2^d}\times ^{D_1} D_2, \; S_4^-\times _{S_4}S_4,\; D_3\times _{D_3}D_3
%\]

On the other hand, one can recognize that the orbit types
\begin{gather*}
(O(2)^-\times D_3), \quad ({O(2)^p}^{O(2)^-}\times ^{D_2}D_4), \quad
({O(2)^p}^{O(2)^-}\times ^{D_1}D_2),\\
 ({D_4^p}^{^{\bz_2^-}}\times _{D_4}D_4), \quad  (S_4^-\times _{S_4}S_4),\quad  (D_3\times _{D_3}D_3)
\end{gather*}
are maximal orbit types  $(\mathscr H')$ in $\mathcal V_1\otimes \cW_4$.
Since $\nabla_{G'}\text{-deg}_{\mathcal{W}_{3}^{-}\otimes\mathcal{W}_{4}}  $ is given by \eqref{eq:g1},  we have
\begin{align*}
\Psi(O(2)^-\times D_3))&= (D_4^z\times D_3)+ (D_2^d\times D_3)+ (D_3^z\times D_3)+\alpha\\
\Psi({O(2)^p}^{O(2)^-}\times ^{D_2}D_4)&=({D_4^p}^{D_4^z}\times ^{D_2} D_4)+
 ({D_2^p}^{D_2^d}\times ^{D_2} D_4)+({D_3^p}^{D_3^z}\times ^{D_2} D_4)+\alpha\\
\Psi( {O(2)^p}^{D_2^d}\times ^{D_1} D_2)&=({D_4^p}^{D_4^z}\times ^{D_1} D_2)+
 ({D_2^p}^{D_2^d}\times ^{D_1} D_2)+({D_3^p}^{D_3^z}\times ^{D_1} D_2)+\alpha
\end{align*}
Consequently,
\begin{align}
\nabla_{G}\text{-deg}_{\mathcal{V}_{1}\otimes\mathcal{W}_{4}}  &
=(G)-({D_{4}^{p}}^{^{{\mathbb{Z}}_{2}^{-}}}\times{_{D_{4}}}D_{4})-(O(2)^-\times D_3))\label{basicdegree_2}
-( {O(2)^p}^{D_2^d}\times ^{D_1} D_2)\\
& -({O(2)^p}^{O(2)^-}\times ^{D_2}D_4)%
-(S_{4}^{-}\times_{S_{4}%
}S_{4})-({D_{3}}\times
{_{D_{3}}}D_{3})+\bm\beta.\nonumber
\end{align}
\end{example}

\end{document}